\documentclass{amsart}
\usepackage{amssymb}
\usepackage{amsmath}
\usepackage{amsfonts}
\usepackage{color}
\usepackage{graphicx}
\usepackage{caption}

\newtheorem{theorem}{Theorem}

\newtheorem{corollary}[theorem]{Corollary}

\newtheorem{definition}[theorem]{Definition}
\newtheorem{example}[theorem]{Example}

\newtheorem{lemma}[theorem]{Lemma}

\newtheorem{proposition}[theorem]{Proposition}
\newtheorem{remark}[theorem]{Remark}

\newtheorem{convention}{Convention}

\begin{document}
\title[Fr\'echet Generalized Trajectories and Minimizers]{Fr\'echet Generalized Trajectories and Minimizers for Variational Problems of Low Coercivity}%
\author[Manuel Guerra \and Andrey
 Sarychev]{Manuel Guerra$^1$ \and Andrey
 Sarychev$^2$}%
\address{$^1$ISEG, University of Lisbon, and CEMAPRE, Portugal,
\newline
$^2$ DiMaI, University of Florence, Italy}%
\email{mguerra@iseg.utl.pt,asarychev@unifi.it}%

\date{}%

\begin{abstract}
We address consecutively two problems. First we introduce a class of so called Fr\'echet generalized controls for a
multi-input control-affine system with non-commuting controlled vector fields.
For each control of the class one is able to define a unique generalized trajectory,and the input-to-trajectory map turns out to be continuous with respect to the Fr\'echet metric.
On the other side, the class of generalized controls is broad enough to settle the second problem, which is proving existence of generalized minimizers of Lagrange variational problem
with functionals of low (in particular linear) growth. Besides we study possibility of Lavrentiev-type gap between the infima of the functionals in the spaces of ordinary and generalized controls.
\end{abstract}
\maketitle

\section{Introduction}
\label{S introduction}

\subsection{Lagrange optimal control problem: classical setting}
\label{SS classical setting}
Consider a Lagrange optimal control problem with control-affine dynamics:
\begin{align}
J(x,u) = & \int_0^1 L(x(t), u (t))   \, dt \rightarrow \min, \label{lag} \\
\notag
\dot x (t) = & f(x(t)) + \sum \limits _{i
=1}^kg_i(x(t))u_i(t)=\\
\label{affine system}
= & f(x(t))+G(x(t))u(t) \qquad \text{a.e. } t \in [0,1], \\
\label{boundary conditions}
 x(0)=&x^0, \qquad x(1) = x^1 .
\end{align}
We assume the vector fields $f$, $g_i, \ i=1,2,\ldots , k$ to be  locally Lipschitz
in $\mathbb R^n$, and the function $(x,u) \mapsto L(x,u)$ to be continuous in $\mathbb R^{n+k}$, and convex with respect to $u$.

Regarding {\it existence of minimizers} for this problem, the  classical approach, pioneered by L.Tonelli and D.Hilbert more  than a century ago, (see monograph \cite{Cesa} for historical remarks and bibliography) introduced the following assumptions for the  Lagrangian $L$
\begin{itemize}
\item[{\bf A)}]
convexity of the Lagrangian  with respect to $u$ for each fixed $x$;
\item[{\bf B)}]
boundedness of the Lagrangian from below and {\it superlinear  growth} of  the Lagrangian as $|u| \to \infty$.
\end{itemize}
Besides, one must require existence of  an admissible  trajectory of  the controlled dynamics  \eqref{affine system} satisfying the boundary conditions  \eqref{boundary conditions}.

These  assumptions guarantee existence of a minimizing control
$\tilde{u}(t) \in L_1^k[0,1]$, which we call {\it ordinary minimizing control}.

It is well known, at least since the work of L.C.Young in the 1930's, that without convexity of $L$ in $u$, ordinary minimizing controls of the Lagrange problem  may cease to exist and the minimum  can be  achieved by so-called {\it relaxed controls}.
By now, a rich theory  of relaxed controls is developed (see \cite{Wa}).

We will not deal with relaxed controls, {\bf assuming below the convexity assumption (A) to hold}.
Instead, we will {\it weaken the condition of superlinear growth} of the Lagrangian,  as $|u| \to \infty$.

\subsection{Weakening growth assumption and generalized minimizers}
\label{SS linear growth}
If instead of  superlinear growth assumption {\bf (B)} we assume  merely
{\it linear growth}
\begin{itemize}
\item[{\bf B$_\ell$)}]$ \ L(x,u)\geq a+ b|u|, \qquad a \in \mathbb{R}, \ b >0$,
\end{itemize}
then existence of ordinary minimizer
for the problem \eqref{lag}--\eqref{boundary conditions} may cease, as the following simple example shows.
\begin{example}[transfer with minimal fuel consumption]
Consider the optimal control problem
\begin{align}
& J(u) = \int_0^1 | u (t)| dt \rightarrow \min, \\
& \dot x (t) = x(t) + u(t), \
  \  u \in \mathbb{R},  \ x(0)=0, \ x(1)=e, \label{km}
\end{align}
which describes transfer of a point  on a line, with the minimized cost, seen as fuel consumption for such transfer.

From the differential system  and the boundary conditions we get
$$e=x(1)=e\int_0^1e^{-\tau}u(\tau)d\tau$$
and then for any control $u(\cdot) \in L_1[0,1]$, compatible with the boundary condition \eqref{km},
$$\int_0^1|u(\tau)|d\tau >   \left|\int_0^1e^{-\tau}u(\tau)d\tau \right|=1. $$
On the other side for a  sequence of needle-like controls
$$u_i(\tau)=\bar{u}_i\chi_{[0,1/i]}(\tau), \qquad  \bar{u}_i=1/(1-e^{-1/i}),$$
which are compatible with the boundary condition,
there holds $J(u_i) \to 1$ as $i \to \infty $.

It is easy to see that the sequence $\{u_i\}$ converges in $W_{-1,1}$-norm to the Dirac measure or optimal {\it impulsive generalized control} $\tilde{u}=\delta(\tau)$;
the corresponding generalized trajectory is a discontinuous function $\tilde{x}(\tau)=e,  \forall \tau >0$, with $\tilde{x}(0)=0. \ \square$
\end{example}

If the Lagrangian $L$ has linear growth with respect to control, then each  sequence of minimizing controls $\{u_i\}_{i \in \mathbb N}$ will be bounded in $L_1$-norm.
If the fields $f, g_1, \ldots g_k$ have linear or sublinear growth with respect to the state variables, then the corresponding sequence of trajectories $\{x_{u_i} \}_{i \in \mathbb N}$ is bounded in total variation.
Helly's selection theorem \cite{Durrett} guarantees that there is a function $x:[0,1] \mapsto \mathbb R^n$ of bounded variation (not necessarily continuous),  such that $x_{u_i}$ converges pointwise to $x$ at every point of continuity of $x$.
It is reasonable to conjecture that there is a space of {\it generalized trajectories} including  discontinuous curves, and a space of {\it generalized controls} including impulses, for which  the problem \eqref{lag}--\eqref{boundary conditions} admits a solution.

\subsection{Generalized and impulsive controls in involutive and non-involutive cases}
\label{SS involutive case}
Study of optimal impulsive controls for linear systems has been initiated in the 1950's, particularly for  applications in spacecraft dynamics.
Later, a more general {\it nonlinear theory} has been developed;  it  englobes the problem \eqref{lag}--\eqref{boundary conditions}
for the cases, where the controlled vector fields $\{g^1, \ldots , g^k\}$ in \eqref{affine system} form an {\it involutive system}.

It turns out that in such cases  one can provide the space of 'ordinary',   controls  $u(\cdot)$ (say
 $L_{1}^{k}[0,T]$) and of the trajectories $x(\cdot)$ with
 weak topologies,  for which one can still guarantee uniform continuity of  the {\it input-to-trajectory map} $u(\cdot)
\mapsto x(\cdot)$. Then one can extend this map by continuity onto the topological completion of the
space of controls, which contains distributions.

Results obtained for {\it nonlinear} control systems by  this approach since the 1970's,  can be found in
\cite{Br87,KrPo,Or,Sa88}.
In particular the method allows to extend the input-to-trajectory map onto the space $W_{-1,\infty}$ of generalized derivatives of measurable essentially bounded functions, with generalized trajectories belonging to $L_\infty$.
Some representation formulae for the generalized trajectories via the generalized primitives of the inputs can be found in \cite{Sa88}.

In the linear-quadratic case, this approach allows for the extension of the input-to-trajectory map \emph{and} the cost functional. Indeed, linear-quadratic Lagrange problems admit a generalized minimizer in some Sobolev space of sufficiently large negative index, provided the boundary conditions can be satisfied and the quadratic functional is bounded from below \cite{Guerra00,ZavalishchinSesekin}.

Problems with the continuous extension of the input-to-trajectory map which arise in the non-involutive  case have been identified in the 1950's (see \cite{Kuzw}).
It has been proved in \cite{KrPo}  that  involutivity of the system of controlled vector fields is necessary for continuity of the map in the weak topology -- a property coined  in \cite{KrPo}  as {\it vibrocorrectness}.

To see, why vibrocorrectness fails in the non-involutive case, look at the following simple example.
\begin{example}
\label{Ex noninvolutive system}
Consider the system
\[
\dot{x}_1=u_1, \ \ \dot{x}_2=u_2, \ \ \dot{x}_3=x_2u_1, \qquad x(0)=(0,0,0) ,
\]
and three bi-dimensional controls, which are  concatenations of needles:
\begin{align*}
&
u^{1,\varepsilon} (t) = \left( \frac 1 \varepsilon \chi_{[0,\varepsilon]}(t), \frac 1 \varepsilon \chi_{[\varepsilon, 2 \varepsilon ]}(t)\right),
\\ &
u^{2,\varepsilon} (t) = \left( \frac 1 \varepsilon \chi_{[0,\varepsilon]}(t),\frac 1 \varepsilon \chi_{[0,\varepsilon ]}(t) \right),
\\ &
u^{3,\varepsilon} (t) = \left( \frac 1 \varepsilon \chi_{[\varepsilon,2\varepsilon]}(t),\frac 1 \varepsilon \chi_{[0,\varepsilon ]}(t) \right).
\end{align*}
For $\varepsilon \rightarrow 0^+$, all the   concatenations tend in $W^2_{-1,1}$ to the bi-dimensional impulsive control
$u(t) = (\delta (t), \delta(t) )$,
while the corresponding trajectories converge pointwise to different discontinuous curves with $x(0^+)=(1,1,0)$, $x(0^+) = \left( 1,1, \frac 1 2 \right)$, and $x(0^+) =(1,1,1)$ respectively. $\square$
\end{example}
Thus, in the noninvolutive case an extension of input-to-trajectory map onto classical spaces of distributions and/or  Sobolev spaces of negative order seems to be impossible.

One approach to the study of noninvolutive systems with impulsive controls proceeds by construction of an appropriate Lie extension of the original system \cite{BressanRampazzo94,Jurdjevic}. The extension is a new system such that: {\it (i)} the extended system of  controlled fields is involutive, {\it (ii)} all the trajectories of the original system are trajectories of the new system, and {\it (iii)} the  trajectories of the extended system can be approximated by trajectories of the original system.
This reduces the noninvolutive case to the involutive and, after some further transformation, to the commutative case.
However, any relation between controls of the extended system and controls of the original system is indirect.

An alternative approach providing a unique extension of the input-to-trajectory map is one of the main issues treated  in  this contribution.

\subsection{Time-reparametrization and "graph completion" techniques in the noncommutative case}
\label{SS noninvolutive case}

For the noncommutative case, a different approach has been adopted.
It is based on a technique of time reparametrization introduced by R.W. Rischel \cite{Rischel65} and J. Warga \cite{Warga65}, and further developed by other authors
\cite{ArutyunovKaramzinPereira10,ArutyunovKaramzinPereira12,
BressanRampazzo88,DykhtaSamsonyuk09,MiRu,MottaRampazzo96,PereiraSilva00,SilvaVinter96,Wa,WargaZhu94}. For a detailed monography and further references, see \cite{MiRu}.
The approach proceeds by introducing a new independent variable with respect to which the trajectories become absolutely continuous.
This creates an auxiliary control system which includes time as an additional state variable.

Several authors  \cite{ArutyunovKaramzinPereira10,ArutyunovKaramzinPereira12,
DykhtaSamsonyuk09,MiRu,MottaRampazzo96,PereiraSilva00,SilvaVinter96}  use the auxiliary system to obtain representations of generalized solutions of \eqref{affine system} by solutions of systems having Radon measures as generalized controls and (right-continuous) functions of bounded variation as generalized trajectories.
The definitions introduced  have a 'sequential form': couples $(x(\cdot), U(\cdot))$ of functions of bounded variation, which are correspondingly the generalized trajectory and the  primitive of generalized control are weak$^*$ limits in {\bf BV} of couples $\left( x_n(\cdot), U_n(\cdot) \right)$ of classical trajectories $x_n$ and primitives $U_n$ of classical controls $u_n$ which generate  $x_n(\cdot)$, with $\sup\limits_n\|u_n(\cdot)\|_{L_1} < \infty$. It is known that in the scope of this approach for the same $U$, different sequences $x_n(\cdot)$, driven by different $U_n$  may converge to different limits,
i.e., each generalized input defines a 'funnel' of generalized trajectories, rather than a well defined unique trajectory.

A different line of argument has been  followed in \cite{BressanRampazzo88}.
Any function $x:[0,1] \mapsto \mathbb R^n$ can be  identified with its {\it graph}, that is the set $\Gamma_x = \left\{ (t,x(t)): t \in [0,1] \right\} \subset \mathbb R^{1+n}$.
If the function $x$ is not continuous, then its graph is not connected. However, if the total variation of $x$ is finite, then there is a {\it graph completion} of $\Gamma_x$ which is connected.
In \cite{BressanRampazzo88}, each control $u \in L_1^k[0,1]$ is identified with the  graph of its primitive $U(t)= \int_0^tu(\tau) d \tau $.
The spaces of generalized controls and generalized trajectories are spaces of graph completions of functions of bounded variation.

The input-to-trajectory map is shown to be continuous over sets of generalized controls equibounded in variation provided with an appropriate metric, into the space of generalized trajectories provided with the Hausdorff metric over the graph completions of generalized primitives.

\subsection{Fr\'echet curves approach to the noncommutative case}
\label{SS new approach}

In what regards the construction of generalized inputs and trajectories, our approach is rather close to the one of  \cite{BressanRampazzo88}.
It is easy to  observe that the graph completions introduced in \cite{BressanRampazzo88} are Fr\'echet curves \cite{Frechet08,Leoni09}, and the metric introduced in the space of generalized controls is the classical Fr\'echet metric.

We prove a stronger version of the main result in \cite{BressanRampazzo88}: the input-to-trajectory map is continuous with respect to a strengthened Fr\'echet metric in both the domain and the image.
Notice that the Fr\'echet metric is topologically stronger than the Hausdorff metric.
Since ordinary controls are densely embedded in the space of Fr\'echet curves, this proves existence and uniqueness of a continuous extension of the input-to-trajectory map into the space of generalized controls.
This map admits a simple representation in the form of an input-to-trajectory map of an equivalent {\it auxiliary system}.
\subsection{Fr\'echet generalized minimizers for Lagrange problems with functionals of linear growth}
\label{SS extended Lagrange problem}
This continuity result together with the representations of generalized trajectories contributes to  proper
extension of the cost functional \eqref{lag} onto the space of Fr\'echet generalized controls.
Thus, we extend the Lagrange variational problem \eqref{lag}--\eqref{boundary conditions} onto the class of Fr\'echet generalized controls and trajectories so that
\begin{itemize}
\item
the cost functional \eqref{lag} is lower semicontinuous in the space of Fr\'echet generalized controls and, under linear growth assumption for the integrand, the problem possesses a Fr\'echet generalized minimizer;
\item
there may exist a (Lavrentiev-type) gap  between the infimum of the cost functional in $L_1^k[0,1]$ and the infimum in the space of generalized controls;
\item
one can formulate regularity conditions which preclude occurrence of the gap.
\end{itemize}

We do not claim finding the weakest topology in the space of controls, which provides continuity of input-to-trajectory map.
In fact, the study in \cite{LiuSussmann99} indicates that, under lack of  {\it involutivity}, the weakest topology should depend on the structure of the Lie algebra generated by the  the vector fields $f, g_1, \ldots, g_k$.
The topology, we introduce does not depend on it.
However, it allows for a proper extension of Lagrange variational problems onto a set of generalized controls, which is broad  enough to guarantee existence of generalized minimizers for integral functionals of low (in particular of linear) growth.

\subsection{Structure of the paper}
\label{SS structure of the paper}
This paper is organized as follows.
In Section \ref{S Generalized trajectories and generalized controls}, we discuss the spaces of Fr\'echet curves and their topologies. We prove the {\it continuous canonical selection} theorem (Theorem \ref{T continuity Frechet to W}), and introduce the spaces of generalized controls and generalized trajectories.
Section \ref{S input-to-trajectory map} deals with the definition of the generalized input-to-trajectory map.
Section \ref{S Problem reduction} discusses the auxiliary problem.
In Section \ref{S the cost functional}, we discuss the extension of the cost functional and its properties.
Existence of minimizers for the extended problems is settled in Section \ref{S existence}.
Possible occurrence  of a Lavrentiev gap is discussed in Section \ref{S Lavrentiev phenomenon}.
In Section \ref{S Example} we present an example of a problem with an integrand of linear growth, whose minimizers are all generalized.
The proofs of some technical results are collected in the appendix (Section \ref{S appendix}).

\section{Fr\'echet generalized controls and generalized paths}
\label{S Generalized trajectories and generalized controls}

The goal of this Section is to introduce the spaces of generalized controls and generalized paths.
Subsection \ref{SS Frechet curves} contains definitions and some basic facts about Fr\'echet curves.
Subsection \ref{SS Continuity of canonical selector} contains the key theorem of {\it continuous canonical selection} (Theorem \ref{T continuity Frechet to W}).
Subsection \ref{SS Frechet curves in space-time}, specialises on Fr\'echet curves defined in space-time.
In Subsections \ref{SS Generalized controls} and \ref{SS Generalized trajectories}, we give definitions of what we call the spaces of {\it Fr\'echet generalized controls and generalized paths}.

\subsection{Fr\'echet curves}
\label{SS Frechet curves}
Various slightly different definitions of Fr\'echet curves can be found in the literature  \cite{Frechet08,Leoni09}.
In this paper we consider curves that are rectifiable and oriented.
Such curves admit absolutely continuous parameterizations, which is a natural requirement when dealing with ordinary differential equations.
We allow Fr\'echet curves to be parameterized by non-compact intervals, which is a convenient way to account for solutions of \eqref{affine system} with a blow up time in the interval $[0,1]$.

Below we state the exact definitions and basic properties.
\medskip

We say that a set $\gamma \subset \mathbb R^n$ is a \emph{parameterized curve} if there is an absolutely continuous function $g:[0,+\infty[ \mapsto \mathbb R^n$ such that $\gamma = g([0,+\infty[)$.
A parameterization provides the curve with a terminal point $g(+\infty)$ only if a finite  limit $\lim\limits_{t \rightarrow +\infty} g(t) $ exists.
In that case we don't distinguish between $\gamma $ and $\gamma \cup \{ g(+\infty) \} $.

\begin{definition}
\label{D orientation}
Two absolutely continuous curves  $g_1,g_2:[0,+\infty[ \mapsto \mathbb R^n$ are equivalent  if
\begin{equation}
\label{Eq orientation}
\inf_{\alpha \in \mathcal T} \left\| g_1-g_2 \circ \alpha \right\|_{L_\infty[0,+\infty [} = 0,
\end{equation}
where $\mathcal T$ denotes  the set of monotonically increasing absolutely continuous bijections $\alpha :[0,+\infty[ \mapsto [0,+\infty[$ admitting absolutely continuous inverse. $\square$
\end{definition}

The following Lemma relates the previous definition with alternative formulations; its proof can be found in Appendix (Subsection \ref{SP L orientation}).

\begin{lemma}
\label{L orientation}
Two absolutely continuous parameterizations  $g_1,g_2:[0,+\infty[ \mapsto \mathbb R^n$ are equivalent
if and only if there are absolutely continuous nondecreasing functions $\alpha_1, \alpha_2: [0,+\infty[ \mapsto[0,+\infty[ $ satisfying the following conditions:
\begin{itemize}
\item[{\bf a)}] $g_1\circ \alpha_1 (t) = g_2 \circ \alpha_2 (t) \qquad \forall t \in [0,+\infty[$;
\item[{\bf b)}] $\alpha_1(0)=\alpha_2(0)=0$ and $\alpha_i([0,+\infty[ ) = [0,+\infty[ $ for at least one $i \in \{1,2\}$;
\item[{\bf c)}] If $\alpha_i (\infty )=T<+\infty $, then $g_i(t)=g_i(T^-)$ for every $t\geq T$. $\square$
\end{itemize}
\end{lemma}

Definition \ref{D orientation} introduces an equivalence relation.
The equivalence class of a  function $g:[0,+\infty[ \mapsto \mathbb R^n$ is
\begin{equation}\label{Fre_class}
[g] = \left\{
h \in AC \left( [0,+\infty[, \mathbb R^n \right):
\inf_{\alpha\in \mathcal T} \left\| g- h \circ \alpha \right\|_{L_\infty[0,+\infty[} = 0
\right\} ,
\end{equation}
and is called \emph{absolutely continuous Fr\'echet curve} in $\mathbb R^n$, or for the sake of brevity, \emph{Fr\'echet curve}; each $\tilde g \in [g]$ will be called either a {\it representative} or {\it parameterization} of $[g]$, depending on context.

The space of  Fr\'echet curves in $\mathbb R^n$ is provided with the \emph{Fr\'echet metric}
\[
d\left( [g_1],[g_2] \right) = \inf_{\alpha \in \mathcal T} \left\| g_1 - g_2 \circ \alpha \right\|_{L_\infty[0,+\infty[} .
\]

Parameterizations by bounded intervals, i.e., absolutely continuous functions $g:[0,T[ \mapsto \mathbb R^n$ or  $g:[0,T] \mapsto \mathbb R^n$,  with $T<+\infty$ are included in the present  definition of Fr\'echet curves:  $[g]$ stays for  $[g \circ \alpha ]$ where $\alpha:[0,+\infty[ \mapsto [0,T[$ is any monotonically increasing absolutely continuous bijection with absolutely continuous inverse.

For every subset $A \subset \mathbb R^n$ and any Fr\'echet curve $[g]$, we set
\[
A \cap [g] = \left\{ g(t): t \in[0,+\infty[, \ g(t) \in A \right\} .
\]
We say that $A \cap [g]$ is a \emph{segment} if the set $\{ t \geq 0 : g(t) \in A \}$ is an interval; according to aforesaid, a nonempty segment is also a Fr\'echet curve.

For every nondecreasing $\alpha:[0,+\infty[ \mapsto [0,+\infty[$, we introduce the function $\alpha^{\#}:[0,+\infty[ \mapsto [0,+\infty]$, defined as
\[
\alpha^{\#}(t) = \sup \{ s\geq 0: \alpha(s)\leq t \} \qquad t \in [0,+\infty[.
\]
If $\alpha$ is continuous, then $\alpha^{\#}$ is the right-inverse of $\alpha$, that is, $\alpha \circ \alpha^{\#}(t) =t$ for every $t< \alpha (+\infty )$.

\begin{lemma}
\label{L AC reparameterization}
Let $\alpha:[0,+\infty[ \mapsto [0,+\infty[$ be nondecreasing absolutely continuous, and $g:[0,+\infty[ \mapsto \mathbb R^n$ be absolutely continuous. Then:
\begin{itemize}
\item[{\bf a)}]
$\dot \alpha \circ \alpha^{\#} (t) >0 $  a.e. on $\left[ \alpha(0), \alpha (\infty) \right[$.
\item[{\bf b)}]
$g \circ \alpha^{\#}$ is absolutely continuous in $\left[ \alpha(0), \alpha (\infty)\right[$ if and only if the set \linebreak $ \left\{ t \geq 0: \dot \alpha (t) =0, \ \dot g (t) \neq 0 \right\}$ has zero Lebesgue measure.
\item[{\bf c)}]
If $g \circ \alpha^{\#}$ is absolutely continuous in $\left[ \alpha(0), \alpha (\infty)\right[$, then
\begin{equation}
\label{Eq derivative of reparameterized curve}
\frac{d}{dt}\left( g \circ \alpha^{\#} \right)(t) = \frac{\dot g}{\dot \alpha} \circ \alpha^{\#} (t) \qquad
\text{for a.e. } t \in \left[ \alpha(0),\alpha (\infty ) \right[ .
\end{equation}
\item[{\bf d)}]
If $g \circ \alpha^{\#}$ is absolutely continuous in
$\left[ \alpha(0), \alpha (\infty)\right[$ and $g$ is constant on each interval
$\left[0,\alpha(0) \right[$,
$ \left] \alpha(\infty),+\infty \right[$, then $g \circ \alpha^\# \in \left[ g \right]$. $\square$
\end{itemize}
\end{lemma}

\begin{proof}
See Subsection \ref{SP L AC reparameterization}.
\end{proof}

An absolutely continuous parameterization $g:[0,+\infty[ \mapsto \mathbb R^n$ generates an arc-length function $\ell_g:[0,+\infty[ \mapsto [0,+\infty[$, defined as
\[
\ell_g(t) = \int_0^t \left| \dot g (s) \right| ds \qquad \forall t \geq 0.
\]

By Lemma \ref{L AC reparameterization}, the function $g \circ \ell_g^{\#} \in [g]$  and $\left| \frac d {dt} \left( g \circ \ell_g^{\#} \right) \right| \equiv 1$.
Further, $\tilde g \circ \ell_{\tilde g}^{\#} = g \circ \ell_g^{\#}$ for every $\tilde g \in [g]$. That is, the transformation $[g] \mapsto g \circ \ell_g^{\#}$ does not depend on the particular $g$ representative of $[g]$.
This transformation selects one particular element of the class $[g]$.
Since it plays an important role in our approach we introduce the following definition:

\begin{definition}
\label{D canonical selector}
We call $g \circ \ell_g^{\#}$ the \emph{canonical representative} or \emph{canonical parameterization} of $[g]$, and the mapping $[g] \mapsto g \circ \ell_g^{\#}$ is the \emph{canonical selector}. $\square$
\end{definition}

The total length of a Fr\'echet curve $[g]$, $\ell_g(\infty)$, does not depend on the particular parameterization $g$.
We call a Fr\'echet curve $[g]$ \emph{infinite} if $\ell_g(\infty)=+\infty$. Otherwise, we call $[g]$ \emph{finite}.
The space of finite Fr\'echet curves can be provided with the \emph{strengthened F\'echet metric}
\[
d^+\left( [g_1],[g_2] \right) =
d\left( [g_1],[g_2] \right) + \left| \ell_{g_1}(\infty ) - \ell_{g_2}(\infty ) \right| .
\]
This metric provides a stronger topology than the Fr\'echet metric,as can be seen from the following example.

\begin{example}
\label{Ex strong Frechet metric}
For the sequence
\[
g_i(t) = \frac 1 i \left( \cos (i^2t), \sin (i^2t) \right) \qquad t \in [0,1], \ i \in \mathbb N ,
\]
$[g_i]$ converges to $[0]$ with respect to the Fr\'echet metric $d$. However $\ell_{g_i}(1) = i$ and therefore $[g_i]$ does not  converge in  the metric $d^+$. $\square$
\end{example}

Every finite Fr\'echet curve $[g]$ has a well defined terminal point $g(\infty)$.
Thus, we adopt the following

\begin{convention}
\label{Cv extension canonical}
Any parameterization of a finite Fr\'echet curve by a compact interval $g:[0,T] \mapsto \mathbb R^n$ is extended  to the interval $[0,+\infty[$ by
\[
g(t) = g(T) \qquad \forall t \geq T.
\]
In particular, a canonical parameterization $t \mapsto g \circ \ell_g^\#(t)$ is extended to the interval $[0,+\infty[$ by
\[
g \circ \ell_g^{\#}(t) = g(\infty) \qquad \forall t \geq \ell_g(\infty) . \ \square
\]
\end{convention}

\subsection{Continuity of canonical selector}
\label{SS Continuity of canonical selector}
The space of Fr\'echet curves consists of sets (classes of curves), in which we introduced the strengthened Fr\'echet metric.
A crucial fact is that choosing the canonical representative of each class provides us with a $C_0$-continuous selector.

\begin{proposition}
\label{P continuity Frechet-to-length}
The canonical selector $[g] \mapsto g \circ \ell_g^{\#}$ is a continuous one-to-one mapping of the space of  finite Fr\'echet curves in $\mathbb R^n$ provided with the strengthened Fr\'echet metric $d^+$ into the space $AC^n[0,+\infty[$ provided with the topology of $C_0$-convergence.  $\square$
\end{proposition}

\begin{proof}
Since for each $g$ the canonical representative $g \circ \ell_g^{\#}$ belongs to $[g]$ and is uniquely defined, then the correspondence $[g] \mapsto g \circ \ell_g^{\#}$ is one-to-one.
It remains to prove that the canonical selector is continuous.

Fix a finite Fr\'echet curve $[ \gamma ]$, with canonical representative $\gamma $, and pick a small $\varepsilon >0$.
There exists a partition $0=t_0<t_1<t_2< \ldots < t_N < +\infty$ such that
\begin{align*}
&
\sum_{i=1}^N \left| \gamma (t_i) - \gamma(t_{i-1}) \right| > \ell_\gamma (\infty )- \varepsilon ,
\qquad
t_N > \ell_\gamma(\infty) + \varepsilon .
\end{align*}
This implies that the length of any segment $\gamma|_{[t_j,t_{j+k}]}$ admits the bounds
\begin{align}
\label{Eq length approximation 2}
\sum_{i=1}^k \left| \gamma (t_{j+i}) - \gamma(t_{j+i-1}) \right| \leq
\ell_{\gamma|_{[t_j,t_{j+k}]} } \leq
\sum_{i=1}^k \left| \gamma (t_{j+i}) - \gamma(t_{j+i-1}) \right|  +  \varepsilon .
\end{align}
Pick an arbitrary Fr\'echet curve $[g]$ such that
\begin{equation}
\label{Z006}
d^+\left( [g],[\gamma] \right) < \frac \varepsilon N ,
\end{equation}
and let $g$ be the canonical representative of $[g]$.
The bound \eqref{Z006} implies
\begin{align}
\label{Eq bounds g gamma}
\left| \ell_g (\infty) - \ell_\gamma (\infty) \right| < \frac \varepsilon N, \qquad
\left\| g \circ \alpha - \gamma \right\|_{L_\infty}< \frac \varepsilon N ,
\end{align}
for some (absolutely continuous monotonically increasing) function $\alpha \in \mathcal T$.
Without loss of generality, we may assume that $\alpha (t) =t $ for every $t \geq t_N > \max \left\{ \ell_g(\infty), \ell_\gamma(\infty) \right\}$.

Let $\theta_i = \alpha (t_i)$ for $i = 0,1,2, \ldots , N$.
By  \eqref{Eq bounds g gamma}, we have
\begin{equation}
\label{Eq bounds g gamma 2}
|g(\theta_i) - \gamma (t_i) | < \frac \varepsilon N \qquad \text{for } i=0,1,2, \ldots , N.
\end{equation}

Let
\[
M=  \left\| g - \gamma \right\|_{L_\infty[0,+\infty[} =
\max \left\{ |g(t) - \gamma (t) |: t \in [0, t_N ] \right\} =
|g(\hat t) - \gamma(\hat t)| .
\]
We may add the point $\hat t$ to the partition and (with a small abuse of notation) think that $\hat t=t_k$ for some $k \in \{0,1,2, \ldots , N \}$.

We obtain
\begin{align*}
\ell_g(\infty) = &
\theta_k +\ell_g(\infty) - \ell_g(\theta_k) =
\theta_k+ \ell_{g |_{[ \theta_k, \theta_N]}} \geq
\theta_k + \sum_{i=k+1}^N \left| g(\theta_i)-g(\theta_{i-1}) \right| \geq
\\ \geq &
\theta_k + \sum_{i=k+1}^N \left(
\left| \gamma(t_i)-\gamma(t_{i-1}) \right| -
\left| g(\theta_i)-\gamma(t_i) \right| -
\left| g(\theta_{i-1})-\gamma(t_{i-1}) \right|
\right) \geq
\end{align*}
and by \eqref{Eq bounds g gamma 2}:
$$ \ell_g(\infty) \geq
\theta_k + \sum_{i=k+1}^N \left| \gamma(t_i)-\gamma(t_{i-1}) \right| - 2 \varepsilon .
$$
By virtue of  \eqref{Eq length approximation 2}, we get
\begin{align}
\ell_g(\infty) \geq &
\theta_k + \ell_{\gamma |_{[t_k,t_N]}} - 3 \varepsilon =
\ell_\gamma(\infty) + \theta_k - t_k -3 \varepsilon .
\label{Z020}
\end{align}
Similar computation, based on  \eqref{Eq bounds g gamma 2} and  \eqref{Eq length approximation 2},  yields
\begin{align}
\notag
\ell_g(\infty) = &
\ell_g(\theta_k) + \ell_g(\infty) - \theta_k \geq
\sum_{i=1}^k \left| g(\theta_i) - g(\theta_{i-1}) \right| + \ell_g(\infty) - \theta_k \geq
\\ \label{Z021}  \geq &
\sum_{i=1}^k \left| \gamma(t_i) - \gamma(t_{i-1}) \right| - 2 \varepsilon + \ell_g(\infty) -\theta_k \geq
\ell_\gamma(\infty) + \ell_\gamma(t_k) - \theta_k - 4 \varepsilon =
\\  \notag = &
\ell_\gamma(\infty) + t_k - \theta_k - 4 \varepsilon .
\end{align}
Joining \eqref{Z020} and \eqref{Z021}, one concludes
\[
|\theta_k - t_k | < \ell_g(\infty) - \ell_\gamma(\infty) + 4 \varepsilon .
\]
Finally, we have the estimate
\begin{align*}
| \theta_k - t_k| = &
\left| \ell_g (\theta_k) - \ell_g(t_k) \right| \geq
\left| g (\theta_k) - g(t_k) \right| =
\left| g \circ \alpha(t_k)- \gamma(t_k) + \gamma(t_k) - g(t_k) \right| \geq
\\ \geq &
\left| \gamma(t_k) - g(t_k) \right| - \left| g \circ \alpha(t_k)- \gamma(t_k) \right| \geq
M - \frac \varepsilon N .
\end{align*}
Thus, \eqref{Z006} implies $\left\| g-\gamma \right\|_{L_\infty}=M < 5 \varepsilon $.
\end{proof}

It is a bit surprising that the continuity property of the canonical selector can be strengthened to $W_{1,p}^n[0,+\infty[$.

\begin{theorem}
\label{T continuity Frechet to W}
The canonical selector $[g] \mapsto g \circ \ell_g^{\#}$ is a continuous map from the space of finite Fr\'echet curves in $\mathbb R^n$ provided with the strengthened Fr\'echet metric $d^+$ into the Sobolev space $W_{1,p}^n[0,+\infty[$, for each  $p \in [1,+\infty[$.  $\square$
\end{theorem}

\begin{remark}
\label{Rm W[0,infty[}
Since we are dealing with finite Fr\'echet curves, each canonical element's derivative is supported in some compact interval.
Therefore, for each pair of finite Fr\'echet curves $[g],[h]$, there is some $T<+\infty$ such that
\[
\left\| g \circ \ell_g^\# - h \circ \ell_h^\# \right\|_{W_{1,p}^n[0,+\infty[}=\left\| g \circ \ell_g^\# - h \circ \ell_h^\# \right\|_{W_{1,p}^n[0,T]} .
\]
However, we cannot fix a priori one such $T$ for every finite
$[g],[h]$. $\square$
\end{remark}

\begin{proof}
Fix $[\gamma]$, a finite Fr\'echet curve with canonical representative  $\gamma$.
 Let a sequence of finite Fr\'echet curves $\{ [ \gamma_i] \}_{i \in \mathbb N}$, with canonical representatives $\gamma_i$, $i \in \mathbb N$, converge to $[\gamma]$:  $\lim\limits_{i \rightarrow \infty} d^+ \left( [\gamma_i], [\gamma] \right) = 0$.

As far as the lengths of $\gamma, \gamma_i, \ i \in \mathbb N$ are bounded by some $T$, then  the interval $[0,T]$ contains the supports of $\dot \gamma, \ \dot \gamma_i$, $i \in \mathbb N$.
According to Proposition \ref{P continuity Frechet-to-length},
$\lim\limits_{i \rightarrow \infty} \|\gamma_i -\gamma\|_{L^n_\infty[0,T]} = 0$.
One wishes to prove  that
$\lim\limits_{i \rightarrow \infty} \left\| \dot \gamma_i - \dot \gamma \right\|_{L^n_p[0,T]}=0,$
and hence
\[
\lim\limits_{i \rightarrow \infty} \|\gamma_i -\gamma\|_{W^n_{1,p}[0,T]} = 0 .
\]

First, we show that $\dot \gamma_i$ converges to  $\dot \gamma$ in the weak$^*$ topology of $L^n_\infty[0,T]$.
Indeed, seeing $\dot \gamma_i$ as a functional on $L^n_1[0,T]$, we note that
\[
\forall t \in [0,T], \ \forall v \in \mathbb R^n: \quad
\langle \gamma_i(t), v \rangle =
\int_0^t \langle \dot \gamma_i(s),v \rangle ds =
\left\langle \dot \gamma_i^j, v \chi_{[0,t]}\right\rangle
\]
is the result of  the action of the functional $\dot \gamma_i$ on the vector-function $s \mapsto v \chi_{[0,t]}(s)$. As far as the space of linear combinations of the functions $v \chi_{[0,t]}(\cdot)$ is dense in $L^n_1[0,T]$, and  $L_\infty$-norms  of  $\dot \gamma_i$ are bounded by $1$,
we conclude that
\[
\lim_{i \rightarrow \infty} \int_0^T \langle \dot \gamma_i , \varphi \rangle dt =
\int_0^T \langle \dot \gamma, \varphi \rangle dt ,
\qquad \forall \varphi \in L^n_1[0,T].
\]

Since $L^n_q[0,T] \subset L^n_1[0,T]$, $\forall q \in [1,+\infty]$, this shows that $\dot \gamma_i \rightharpoondown \dot \gamma$ weakly in $L^n_p[0,T]$ for every $p \in ]1, +\infty[$.

Note that from  $d^+$-convergence of $\gamma_i$ to $\gamma$, it follows that
\[
\lim \left\| \dot \gamma_i \right\|_{L^n_p[0,T]} =
\lim \ell_{\gamma_i}(\infty) =
\ell_{\gamma}(\infty) =
\left\| \dot \gamma \right\|_{L^n_p[0,T]} \qquad \forall p \in ]1,+\infty[.
\]
Therefore, the Radon-Riesz theorem \cite{Riesz-Nagy} guarantees that
\[
\lim \left\| \dot \gamma_i - \dot \gamma \right\|_{L^n_p[0,T]} =0 \qquad \forall p \in ]1,+\infty[ ,
\]
which implies $\lim \left\| \dot \gamma_i - \dot \gamma \right\|_{L^n_1[0,T]} =0$.
\end{proof}

\subsection{Fr\'echet curves in space-time}
\label{SS Frechet curves in space-time}
Let $\mathcal Y_n$ denote the set of absolutely continuous functions $(\theta,y):[0,+\infty[ \mapsto \mathbb R^{1+n}$ such that
\begin{align}
\label{Eq time monotonicity} &
\theta (0) =0, \qquad \dot \theta (t) \geq 0 \quad \text{a.e. } t \geq 0 .
\\ \label{Eq infinite length} &
\ell_{(\theta,y)}(\infty) = + \infty .
\end{align}
The first coordinate $\theta$ represents time. Thus, each $(\theta,y) \in \mathcal Y_n$ is a parameterization of a curve in space-time, defined in the time interval $[0,\theta(\infty)[$.

The condition \eqref{Eq time monotonicity} reflects the fact that time should be a monotonically increasing variable. We don't require it to be strictly increasing because we are interested in jumps and impulses, i.e., processes that evolve instantaneously.

The condition \eqref{Eq infinite length} means that the time interval $[0,\theta(\infty)[$ is maximal, that is, the curve parameterized by $(\theta, y)$ cannot be prolonged beyond the time $\theta(\infty) \in [0,+\infty]$.

For each $T >0$, let $\mathcal Y_{n,T}$ be the set of all $(\theta,y) \in \mathcal Y_n$ such that $\theta(\infty ) >T$,
 i.e. the set of all absolutely continuous parameterizations of curves in space-time, which  are well defined on the compact  time interval $[0,T]$.

For each $(\theta,y) \in \mathcal Y_n $, define $(\theta_T, y_T)$ as
\begin{equation}
\label{Eq (vT,yT)}
\begin{array}{ll}
(\theta_T,y_T)(t) = (\theta , y )(t) & \text{for } t \in [0, \theta^\#(T) [, \smallskip \\
(\theta_T,y_T)(t) = \left(T , y(\theta^\#(T)\right) & \text{for } t \geq \theta^\#(T) ;
\end{array}
\end{equation}
in particular $(\theta_T,y_T)$ coincides with  $ (\theta , y ) $ if $\theta^\#(T)=+\infty$.

Now we introduce a family of semimetrics in $\mathcal Y_n$, $\{ \rho_T \}_{T \in ]0,+\infty[}$, defined as
\[
\rho_T \left((\theta,y), (\tilde \theta,\tilde y) \right) =
\left\| (\theta_T, y_T)- ( \tilde \theta_T, \tilde y_T)  \right\|_{L_\infty[0,+\infty[} ,
\]

Each $\rho_T$ becomes a metric if we don't distinguish between $(\theta,y),(\tilde \theta, \tilde y) \in \mathcal Y_n$ such that $\theta^{\#}(T)= \tilde \theta^{\#}(T)$ and $(\theta, y)$ coincides with $(\tilde \theta, \tilde y)$ in $[0,\theta^{\#}(T)[$.

\begin{convention}
\label{Cv extension Y}
When dealing with  the metric $\rho_T$, we identify  $(\theta,y) \in \mathcal Y_n$ with the corresponding $(\theta_T,y_T)$, defined in \eqref{Eq (vT,yT)}. $\square$
\end{convention}

Let $\mathcal F_n$ and $\mathcal F_{n,T}$ be the spaces of Fr\'echet curves in $\mathbb R^{1+n}$ corresponding to $\mathcal Y_n$ and $\mathcal Y_{n,T}$, that is
\[
\mathcal F_n = \left\{ [(\theta,y)]: (\theta,y ) \in \mathcal Y_n \right\} ,
\qquad
\mathcal F_{n,T} = \left\{ [(\theta,y)]: (\theta,y ) \in \mathcal Y_{n,T} \right\} .
\]
Due to condition \eqref{Eq infinite length}, $\mathcal F_n$ is a set of infinite Fr\'echet curves.

Each semimetric $\rho_T$ induces a semimetric $d_T$ in the space $\mathcal F_n$:
\begin{align*}
&
d_T \left( [(\theta_1,y_1)], [(\theta_2,y_2)] \right) =
\\ = &
\inf \left\{
\rho_T \left( (\tilde \theta_1,\tilde y_1), (\tilde \theta_2,\tilde y_2) \right):
(\tilde \theta_i,\tilde y_i) \in [(\theta_i,y_i)], \ i=1,2
\right\}
\end{align*}

By Convention \ref{Cv extension Y}, $d_T$ becomes a metric;  $d_T\left( [(\theta_1, y_1)], [(\theta_2, y_2)] \right)$ coincides with the Fr\'echet distance between the segments $[(\theta_1, y_1)] \cap \left( [0,T] \times \mathbb R^n \right)$ and $[(\theta_2, y_2)] \cap \left( [0,T] \times \mathbb R^n \right)$, for any $[(\theta_1, y_1)], [(\theta_2, y_2)] \in \mathcal F_{n}$.
When dealing with the metric $d_T$ we identify each $[(\theta,y)] \in \mathcal F_n$ with the segment $[(\theta, y)] \cap \left( [0,T] \times \mathbb R^n \right)$.
\medskip

Consider an absolutely continuous function $x:[0,T[ \mapsto \mathbb R^n$, where $T \in ]0,+\infty]$ is maximal in the sense that $x$ does not admit an absolutely continuous extension onto any interval $[0,\hat T[$ with $\hat T>T$.
Then, the function $t \in [0,T[ \mapsto (t,x(t)) \in \mathbb R^{1+n}$ is an element of $ \mathcal Y_n$ and hence $[(t,x)]$ is an element of $\mathcal F_n$.
The correspondences $x \mapsto (t,x)$ and $x \mapsto [(t,x)]$ are one-to-one.

Conversely, the Lemma \ref{L orientation} implies that for every $[(\theta, y)] \in \mathcal F_n$, the function $y \circ \theta^\#: \left[ 0, \theta (\infty) \right[ \mapsto \mathbb R^n $ does not depend on the particular  representative of $[(\theta,y)]$. In particular, $y \circ \theta^\# = x$ for every $(\theta, y) \in [(t,x)]$.
Following this argument, we identify each absolutely continuous function $x $ defined on a maximal interval with the corresponding Fr\'echet curve $[(t,x)]$.

For every $[(\theta, y)] \in \mathcal F_n$, the function $y \circ \theta^\#$ has bounded variation on compact subintervals of $[0, \theta (\infty)[$.
Due to Lemma \ref{L AC reparameterization}, $y \circ \theta^\#$ is absolutely continuous if and only if the set $\left\{ t \geq 0: \dot \theta (t) =0 , \ \dot y(t) \neq 0 \right\}$ has zero Lebesgue measure.

The following proposition shows that every function of locally bounded variation can be 'lifted' to a Fr\'echet curve by virtue of the transformation $[(\theta,y)] \mapsto y \circ \theta^\#$.

\begin{proposition}
\label{P lift of functions of bounded variation}
If a function $x:[0,T[ \mapsto \mathbb R^n$ has finite variation on every compact subinterval of $[0,T[$, then there exists  $(\theta,y) \in \mathcal Y_n$ such that $x(t) = y \circ \theta^\#(t)$ for every $t \in [0,T[$, a continuity point of $x$. $\square$
\end{proposition}

\begin{proof}
See Appendix (Subsection \ref{SP P lift of functions of bounded variation}).
\end{proof}

Note that the mapping $[(\theta,y)] \mapsto y \circ \theta^\#$ is not one-to-one. Indeed for each function of locally bounded variation $x:[0,T[ \mapsto \mathbb R^n$, there are infinitely many $[(\theta,y)] \in \mathcal F_n$, such that $y\circ \theta^\#=x$. To see this, consider the following example.

\begin{example}
\label{Ex lift of discontinuous curve}
For a discontinuous  function $x:[0,+\infty[ \mapsto \mathbb R^2$ defined as
\[
\begin{array}{ll}
x_1(t)=x_2(t)=0, & \text{for } t<1, \smallskip \\
x_1(t)=0, \ \ x_2(t)=1, & \text{for } t\geq 1.
\end{array}
\]
Every Fr\'echet curve $[(\theta,y)]$ in space-time,  consisting of the concatenation of the following arcs
\begin{itemize}
\item[{\bf i.}]
the segment of straight line from the point $(0,0,0)$ to the point $(1,0,0)$;
\item[{\bf ii.}]
an absolutely continuous Fr\'echet curve from the point $(1,0,0)$ to the point $(1,0,1)$, contained in the plane $\left\{ (1,x_1,x_2): (x_1,x_2) \in \mathbb R^2 \right\}$;
\item[{\bf iii.}]
the ray $\left\{ (1+t,0,1), t \geq 0 \right\}$
\end{itemize}
satisfies $y \circ \theta^\# (t) =x(t) \ \forall t \geq 0$ (see Figure \ref{Fg lifts of function}). $\square$
\begin{figure}[h] \vspace{-0.7cm}
\begin{center}
\begin{minipage}[t]{12cm}
\includegraphics*[height=10cm]{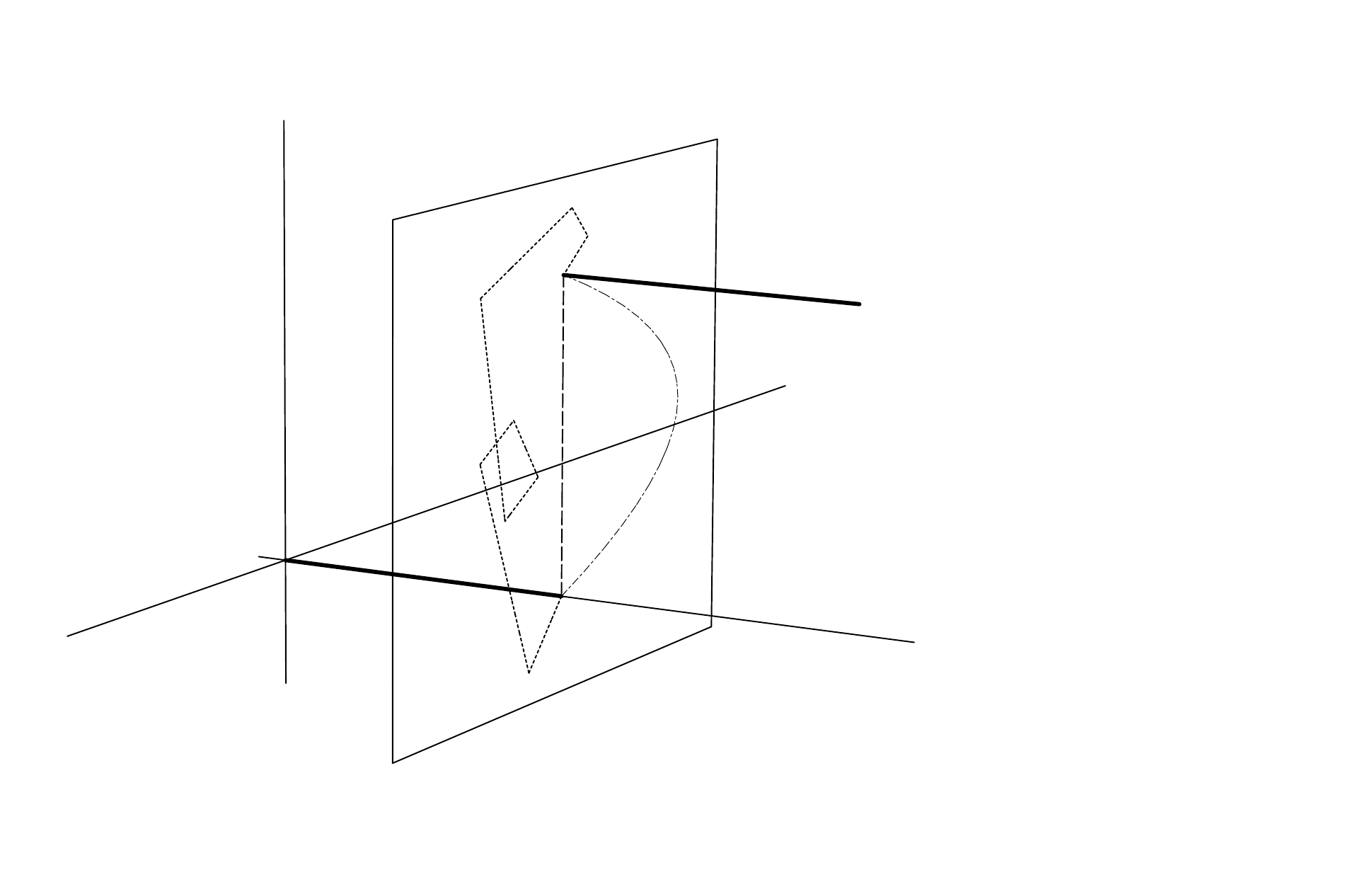}

\vspace{-2.8cm} \hspace{10cm} $t$

\vspace{-3.4cm} \hspace{8.5cm} $x_1$

\vspace{-3.4cm} \hspace{3.2cm} $x_2$

\vspace{5.9cm} \hspace{4.5cm} $t=1$

\end{minipage}
\end{center}
\vspace{-1.5cm}

\caption{
\label{Fg lifts of function}
The discontinuous function $x$ (solid line) and some Fr\'echet curves satisfying $y\circ \theta ^\#=x$ (different dashed lines).
}
\end{figure}
\end{example}

\begin{remark}
\label{Rm graphs}
The \emph{graph} of a function $x:[0,T[ \mapsto \mathbb R^n$ is the set $\Gamma_x = \left\{ (t,x(t)): t \in [0,T[ \right\}$.
If $x(t)$ is absolutely continuous, then the function $t \mapsto (t,x(t))$ is  absolutely continuous parameterization of $\Gamma_x$.
Thus, Fr\'echet curves in space-time can be seen as  generalizations  of absolutely continuous graphs to the class of functions of locally bounded variation. $\square$
\end{remark}

Fr\'echet curves in space-time coincide with {\it graph completions} in the terminology of \cite{BressanRampazzo88}.

\subsection{Fr\'echet generalized controls}
\label{SS Generalized controls}

By  \emph{ordinary control}, we understand any locally integrable function $u: [0,+\infty[ \mapsto \mathbb R^k$.
The control system \eqref{affine system} can be represented as
\begin{equation}
\label{Eq affine system U}
\dot x(t)= f(x(t)) + G(x(t)) \dot U (t) ,
\end{equation}
where $U(t) = \int _0^t u(s) ds $.

The function $t \mapsto (t,U(t))$ is a representative of the Fr\'echet curve $[(t,U)]$. By construction $U(0)=0$; besides local integrability of $u$ guarantees that $(t,U) \in \mathcal Y_{k,T}$, and thus $[(t,U)] \in \mathcal F_{k,T}$, for every $T\in ]0,+\infty[$.

For each $T \in ]0,+\infty[$, we introduce the set
\[
\mathcal Y_{k,T}^0 = \left\{ (V,W) \in \mathcal Y_{k,T}: W(0)=0 \right\} ,
\]
and define the space of \emph{Fr\'echet generalized controls} on the interval $[0,T]$ to be the set of  Fr\'echet curves, whose  representatives belong to   $\mathcal Y_{k,T}^0$, that is
\[
\mathcal F_{k,T}^0 = \{ [(V,W)] \in \mathcal F_{k,T}: W(0)=0 \}
.
\]

According to Convention \ref{Cv extension Y}, we identify each generalized control $[(V,W)] \in \mathcal F_{k,T}^0$ with it's segment $[(V,W)] \cap \left( [0,T] \times \mathbb R^k \right)$, and provide this space with the strengthened Fr\'echet metric:
\begin{align*}
d_T^+\left( [(V_1,W_1)], [(V_2,W_2)] \right) = &
d_T  \left( [(V_1,W_1)], [(V_2,W_2)] \right) +
\\ & +
\left| \ell_{(V_1,W_1)} (V_1^{\#}(T)) -  \ell_{(V_2,W_2)} (V_2^{\#}(T)) \right| .
\end{align*}

A Fr\'echet curve $[(V,W)] \in \mathcal F_{k,T}^0$ coincides with an ordinary control in the interval $[0,T]$ if and only if the set $\{ t: \dot V(t) = 0 \ \text{and} \ \dot W(t) \neq 0 \}$ has zero Lebesgue measure. In that case, $U=W\circ V^\#$ is absolutely continuous, and the corresponding ordinary control is $u(t) = \frac{d}{dt}\left( W \circ V^\# \right)(t)$.

The following proposition demonstrates that every generalized control can be approximated by sequences of ordinary controls.
\begin{proposition}
\label{P density of ordinary controls}
The space of ordinary controls $\left\{ \left[(t,\int_0^t u(s) ds ) \right] : u \in L_1^k[0,T] \right\} $ is dense in $\left( \mathcal F_{k,T}^0, d_T^+\right) $. $\square$
\end{proposition}

\begin{proof}
We introduce in $\mathcal Y_{k,T}$ a semimetric $\rho_T^+$:
\begin{align}
\notag
\rho_T^+ \left( (V_1,W_1), ( V_2, W_2) \right) = &
\rho_T \left( (V_1,W_1), ( V_2, W_2) \right) +
\\ \label{Eq RhoPlus} & +
\left| \ell_{(V_1,W_1)} (V_1^{\#}(T)) -  \ell_{(V_2,W_2)} (V_2^{\#}(T)) \right| .
\end{align}

Fix arbitrary $(V,W ) \in \mathcal Y_{k,T}^0$. For each $\varepsilon>0$, let
\[
V_\varepsilon(t) = \left\{
\begin{array}{ll}
\frac{T}{\int_0^{ V^{\#}(T)} \max ( \dot V, \frac{\varepsilon}{T}) d\tau}
\int_0^t \max ( \dot V, \frac{\varepsilon}{T}) d\tau ,
&
\text{for } t \leq  V^\#(T),
\smallskip \\
T+t- V^\#(T)
&
\text{for } t > V^\#(T) .
\end{array}
\right.
\]
$V_\varepsilon $ admits absolutely continuous inverse and therefore $U_\varepsilon (t) = W \circ V_\varepsilon^{-1}(t)$ is absolutely continuous with $U_\varepsilon(0)=0$.

   One can check that $V_\varepsilon^{-1}(T)= V^{\#}(T) $, $V_\varepsilon$ converges to $V$ uniformly in  $[0,V^\#(T)]$, as  $\varepsilon \rightarrow 0^+$, and
$$\lim\limits_{\varepsilon \rightarrow 0^+} \ell_{(V_\varepsilon,W)}(V^\#(T)) = \ell_{(V ,W)}(V^\#(T)) . $$
Therefore,
$\lim\limits_{\varepsilon \rightarrow 0^+} \rho_T^+ \left( (V_\varepsilon,W), ( V, W) \right)=0 , $
which implies
$
\lim\limits_{\varepsilon \rightarrow 0^+} d_T^+ \left( [(t,U_\varepsilon), [( V, W)] \right)=0
.$
\end{proof}

The space of generalized controls has the following compactness property:
\begin{proposition}
\label{P compactness}
Every sequence of ordinary controls bounded in $L_1^k[0,T]$ admits a subsequence $\{u_i\}_{i \in \mathbb N}$ such that
$\left\{ \left[ (t,U_i) \right] \right\}_{i \in \mathbb N}$ converges in $\left( \mathcal F_{k,T}^0, d_T^+ \right)$. $\square$
\end{proposition}

\begin{proof}
Let a sequence $\left\{ u_i \in L_1^k[0,T] \right\}_{i \in \mathbb N}$ be such that
\begin{equation}
\label{Z005}
\|u_i \|_{L_1^k[0,T]} \leq M, \qquad \forall i \in \mathbb N .
\end{equation}
For each $i \in \mathbb N$, let $(V_i,W_i)= \left( \ell_{(t,U_i)}^{-1},U_i \circ \ell_{(t,U_i)}^{-1} \right) $ be the canonical parameterization of $\left[ (t,U_i) \right]$.
By \eqref{Z005}, $\ell_{(t,U_i)}(T) \leq M+T, \ \forall i \in \mathbb N$.
The sequence $\left\{ (V_i,W_i) \right\}_{i \in \mathbb N}$ is uniformly bounded and equicontinuous in  $[0,T+M]$. Therefore by the Ascoli-Arzel\`a theorem, it admits a uniformly converging subsequence. It follows that the corresponding subsequence $\left\{ \left[ (V_{i_j},W_{i_j} ) \right] \right\}_{j \in \mathbb N}$ converges in $\left( \mathcal F_{k,T}^0, d_T^+ \right)$.
\end{proof}

Let us revise Example \ref{Ex noninvolutive system}.

\begin{example}
\label{Ex noninvolutive system c1}
Consider the controls $u^{i,\varepsilon}, \ i=1,2,3$ from Example \ref{Ex noninvolutive system}.
Let $U^{i,\varepsilon}$ be the respective primitives.
It can be shown that the limits
\[
\left[ (V_i,W_i) \right] = \lim_{\varepsilon \rightarrow 0^+} \left[ (t,U^{i,\varepsilon}) \right] , \qquad i=1,2,3,
\]
exist in $\left( \mathcal F_{n,T}^0, d_T^+ \right)$, and differ on the segment from the point $(0,0,0) $ to the point $(0,1,1)$, as shown on Figure \ref{Fg generalized controls}.
\begin{figure}[h] \vspace{-0.7cm}
\begin{center}
\begin{minipage}[t]{12cm}
\includegraphics*[height=10cm]{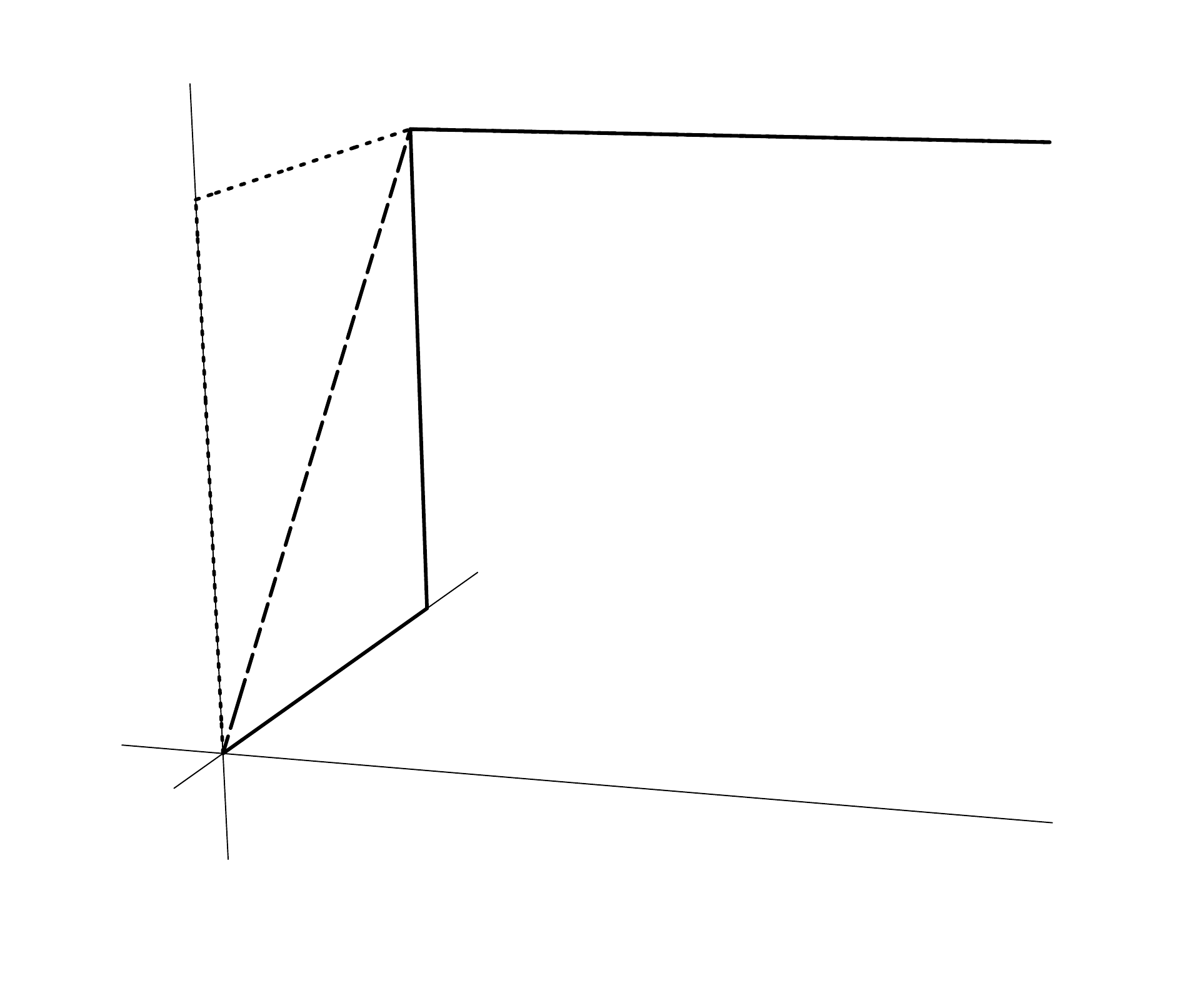}

\vspace{-2.5cm} \hspace{10cm} $t$

\vspace{-2.3cm} \hspace{4.6cm} $U_1$

\vspace{-5.4cm} \hspace{1.2cm} $U_2$

\end{minipage}
\end{center}
\vspace{-1.5cm}

\caption{
\label{Fg generalized controls}
The generalized controls from Example \ref{Ex noninvolutive system c1}: $\left[ (V_1,W_1) \right]$ (solid line), $\left[ (V_2,W_2) \right]$ (dashed line), and $\left[ (V_3,W_3) \right]$ (doted line).
}
\end{figure}
One can directly  compute
\begin{align*}
&
d_T^+ \left( \left[ (V_1,W_1) \right], \left[ (V_2,W_2) \right] \right) =
d_T^+ \left( \left[ (V_3,W_3) \right], \left[ (V_2,W_2) \right] \right) =
2 - \frac{1}{\sqrt{2}},
\\ &
d_T^+ \left( \left[ (V_1,W_1) \right], \left[ (V_3,W_3) \right] \right) = 1 . \ \square
\end{align*}
\end{example}

The ability to characterize impulses  by a path in space-time  provides  a "resolution" of an impulse, which  takes place in zero time, as Example \ref{Ex noninvolutive system c1} illustrates. This feature is crucial for dealing  with the  case,  where the controlled vector fields in the system \eqref{affine system} do not commute.

\subsection{Fr\'echet generalized paths}
\label{SS Generalized trajectories}

We call \emph{ordinary path} any absolutely continuous function $x:[0,T[ \mapsto \mathbb R^n$, with  $T \in ]0,+\infty]$ being  maximal in the sense that $x$ does not admit an absolutely continuous extension onto  any interval $[0,\hat T[\supset [0,T]$.

Using the notation of Section \ref{SS Frechet curves in space-time}, $x \mapsto (t,x)$ and $x \mapsto [(t,x)]$ are one-to-one mappings from the space of ordinary trajectories into $\mathcal Y_n$ and $\mathcal F_n$, respectively.

We identify each ordinary path  $x$ with the corresponding Fr\'echet curve $\left[ (t,x) \right] \in \mathcal F_n$, and define the space of \emph{generalized paths} to be the space $\mathcal F_n$,  provided with the metric $d_T$ (with $T\in ]0,+\infty[$ fixed).

Since the metric $d_T$ is fixed, we follow Convention \ref{Cv extension Y} and identify each generalized path $[(\theta,y)] \in \mathcal F_n$ with its segment $[(\theta,y)] \cap [0,T] \times \mathbb R^n$.
In particular, any ordinary path $x:[0,\tilde T[ \mapsto \mathbb R^n$, with $\tilde T>T$, is identified with its restriction to the interval $[0,T]$.

According to Section \ref{SS Frechet curves in space-time}, $\mathcal F_{n,T}$ is the space of  generalized paths defined on the time interval $[0,T]$ and extendable  beyond $[0,T]$.
Taking into account the conventions above, for any ordinary path $x$, $(t,x) $ is a representative of some $[(\theta,y)] \in \mathcal F_{n,T}$ if and only if $x \in AC^n[0,T]$.
That is, we identify the space of ordinary paths that can be extended beyond $[0,T]$ with the set
\[
\left\{ [(t,x)]: x \in AC^n[0,T] \right\} \subset \mathcal F_{n,T} .
\]

For generalized paths, we have an analogue of  Proposition \ref{P density of ordinary controls}:

\begin{proposition}
\label{P density of ordinary trajectories}
The space of ordinary paths defined in the interval $[0,T]$, $\{[(t,x)] : x \in AC^n[0,T] \}$ is dense in $\left( \mathcal F_{n,T}, d_T^+ \right)$ (and therefore, it is also dense in $\left( \mathcal F_{n,T}, d_T \right)$). $\square$
\end{proposition}

\begin{proof}
A trivial adaptation of the proof of Proposition \ref{P density of ordinary controls}.
\end{proof}

\section{The input-to-trajectory map}
\label{S input-to-trajectory map}

This section is centered on  Theorem \ref{T extension of input-to trajectory map}, which establishes existence and uniqueness of a continuous extension of the input-to-trajectory map onto the space of Fr\'echet generalized controls. The extended map  takes values in  the space of generalized paths.
\medskip

We use capital letters for the  elements of $\mathcal Y_{k,T}^0$ and small letters for  their first-order derivatives (i.e., $(v,w)=(\dot V, \dot W)$, with $(V,W) \in \mathcal Y_{k,T}^0$).

For each locally integrable $u :[0,+\infty [ \mapsto \mathbb R^k$, let $x_u$ denote the corresponding trajectory of the system \eqref{affine system}, starting at a point $x_u(0)=x^0$, defined on a  maximal interval.
We will show that the input-to-trajectory mapping $u \mapsto x_u$ defines a unique mapping $[(t,U)] \mapsto [(t,x_u)]$, which in its turn  admits a unique continuous extension onto $\mathcal F_{k,T}^0$.

There is the following  intuition behind this.
Fix an ordinary control $u$, and suppose that $x_u$ is well defined in the interval $[0,T]$. Pick a  representative  $(V,W) \in [(t,U)]$, and let $y=x_u \circ V$. It follows that
\begin{align}
\label{Eq reduced system y}
y(0)=x_0, \qquad \dot y(t) = f(y(t))v(t) + G(y(t))w(t) \quad \text{a.e. } t \in [0, V^\#(T)].
\end{align}
Thus, if we denote by $y_{(v,w)}$ the unique trajectory of \eqref{Eq reduced system y}, then one may expect that
$(V,y_{(v,w)}) \in [(t,x_u)]$ for every $(V,W) \in [(t,U)]$.

The following Proposition shows that this is also true for  generalized controls, thus providing a "natural" extension of the input-to-trajectory map.

\begin{proposition}
\label{P parameterization invariance of trajectories}
The mapping
\begin{align}
\label{Eq input-to-trajectory map}
[(V,W)] \in \mathcal F_{k,T}^0 \mapsto \left[ (V,y_{(v,w)}) \right] \in \mathcal F_n
\end{align}
is properly defined, i.e., it does not depend on the choice of $(V,W) \in [(V,W)]. \
\square$
\end{proposition}

\begin{proof}
Fix a generalized control $[(\hat V,\hat W)] \in \mathcal F_{k,T}^0$ and the  canonical representative  $(\hat V, \hat W)$. For any $(V,W) \in [(\hat V, \hat W)]$,
\[
(\hat V, \hat W) = (V,W) \circ \ell_{(V,W)}^{\#},
\qquad
(\hat v, \hat w) = \frac{(v,w)}{\sqrt{v^2+|w|^2}} \circ \ell_{(V,W)}^{\#} ,
\]
and it follows that
\[
\left( V, y_{(v,w)} \right) \circ \ell_{(V,W)}^{\#} =
\left( V\circ \ell_{(V,W)}^{\#} , y_{(v,w)} \circ \ell_{(V,W)}^{\#}  \right) =
\left( \hat V , y_{(v,w)} \circ \ell_{(V,W)}^{\#}  \right) .
\]
Since $\{ t: (v(t),w(t))=0 \wedge \dot y_{(v,w)}(t) \neq 0 \} $ has zero Lebesgue measure, the Lemma \ref{L AC reparameterization} guarantees that the function $t \mapsto y_{(v,w)} \circ \ell_{(V,W)}(t)$ is absolutely continuous and
\begin{align*}
&
\frac d {dt}\left( y_{(v,w)} \circ \ell_{(V,W)}^{\#}  \right) =
\left( \dot y_{(v,w)} \circ \ell_{(V,W)}^{\#}  \right) \frac{1}{\sqrt{v^2+|w|^2}} \circ \ell_{(V,W)}^{\#} =
\\ = &
f\left( y_{(v,w)} \circ \ell_{(V,W)}^{\#} \right) \frac{v}{\sqrt{v^2+|w|^2}} \circ \ell_{(V,W)}^{\#} +
G\left( y_{(v,w)} \circ \ell_{(V,W)}^{\#} \right) \frac{w}{\sqrt{v^2+|w|^2}} \circ \ell_{(V,W)}^{\#} =
\\ = &
f\left( y_{(v,w)} \circ \ell_{(V,W)}^{\#} \right) \hat v +
G\left( y_{(v,w)} \circ \ell_{(V,W)}^{\#} \right) \hat w ,
\end{align*}
 then $y_{(\hat v, \hat w)}= y_{(v,w)}\circ \ell_{(V,W)}^{\#} $, and therefore, $\left[(V,y_{(v,w)}) \right] = \left[ (\hat V, y_{\hat v, \hat w}) \right ]$.
\end{proof}

To formulate the result on continuity of the input-to-trajectory map we introduce the set
\[
\mathcal W_T = \left\{
[(V,W)] \in \mathcal F_{k,T}^0: [(V, y_{(v,w)})] \in \mathcal F_{n,T}
\right\} ,
\]
of  generalized controls,  such that the generalized trajectory assigned to them by \eqref{Eq input-to-trajectory map} is well defined on the time interval $[0,T]$.

The following is a stronger version of Theorem 2, Corollary 1 in \cite{BressanRampazzo88}.

\begin{theorem}
\label{T extension of input-to trajectory map}
The set $\mathcal W_T$ is an open subset of $\left( \mathcal F_{k,T}^0, d_T^+ \right)$.

The mapping $[(V,W)] \in \mathcal W_T \mapsto \left[ (V,y_{(v,w)}) \right] \in \mathcal F_{n,T}$ is the unique extension of the input-to-trajectory map that is continuous with respect to the metrics $d_T^+$ in the domain and in the image. $\square$
\end{theorem}

\begin{proof}
The transformation $[( V, W)] \mapsto \left[ (  V,y_{( v, w)}) \right] $ can be decomposed into a chain of mappings
\begin{align}
\notag
[(V, W)] \in \mathcal F_{k,T}^0 \mapsto &
(V, W) \in W_{1,1}^{1+k}[0,+\infty[ \mapsto
\\ \mapsto & \label{Eq transformation chain}
(V,y_{(v,w)}) \in \mathcal Y_{n,T} \mapsto
[(V,y_{(v,w)})] \in \mathcal F_{n,T} ,
\end{align}
where the first transformation is the canonical selector, which is continuous  by Theorem \ref{T continuity Frechet to W}.

We provide the space $\mathcal Y_{n,T}$ with the metric $\rho_T^+$ defined in \eqref{Eq RhoPlus}.
By definition, $d_T^+\left( [(\theta_1,y_1)],[(\theta_2,y_2)] \right) \leq \rho_T^+ \left( (\theta_1,y_1) , (\theta_2,y_2)  \right) $ for every $(\theta_1,y_1) $, $(\theta_2,y_2) \in \mathcal Y_{n,T}$. Hence, the last transformation in \eqref{Eq transformation chain} is continuous.

Fix a generalized control $[(\hat V,\hat W)] \in \mathcal W_T$.
Under Conventions \ref{Cv extension canonical} and \ref{Cv extension Y}, the support of $(\hat v,\hat w)$ is contained in the compact interval $[0,\hat V^\#(T)]$, and for every $[( V, W)] \in \mathcal F_{k,T}^0 $ such that $ d_T^+ \left( [( V, W)], [(\hat V, \hat W)] \right) < \varepsilon $, the support of $( v,  w)$ is contained in $[0, \hat V^\#(T)+\varepsilon]$

By standard continuity result, there is some $\varepsilon>0$ such that the trajectory of the system \eqref{Eq reduced system y} is well defined on the interval $[0,\hat V^\#(T)+\varepsilon]$ for every $( V,  W) \in W_{1,1}^{1+k}[0,+\infty[$ such that $\left\| ( v, w) - (\hat v, \hat w) \right\|_{L_1^{1+k}[0,+\infty[} < \varepsilon$.

Since the input-to-trajectory map of system \eqref{Eq reduced system y} $(v,w) \mapsto y_{(v,w)}$ is continuous with respect to the norms $\| \cdot \|_{L_1^{1+k}[0, \hat V^\#(T) + \varepsilon]}$ in the domain and $\| \cdot \|_{W_{1,1}^{1+n}[0, \hat V^\#(T) + \varepsilon]}$ in the image, the Theorem follows.
\end{proof}

We compare Theorem \ref{T extension of input-to trajectory map} with the corresponding result (Theorem 2, Corollary 1) of \cite{BressanRampazzo88}.
There, the generalized controls have equibounded variations and the metric in the space of impulses is (in our terminology) the Fr\'echet metric.
The topology in the space of generalized trajectories is defined by the Hausdorff metric on the graphs of generalized trajectories (i.e., in the images of the curves $(\theta (t), y(t) )$, $t \in [0, \theta^\#(T)]$).

By introducing the strengthened Fr\'echet metric $d^+$, we automatically require convergence of the full variations, and therefore guarantee equiboundedness of converging sequences of inputs.
The main difference lies in the fact that we prove continuity of the input-to-trajectory map when the space of generalized trajectories is provided with the $d^+$ metric instead of the weaker Hausdorff metric.
Continuity of the canonical selector (Theorem \ref{T continuity Frechet to W}) is essential for this result.

To see that the topology generated by $d^+$ is strictly stronger than the topology generated by the Hausdorff metric, consider the following simple example:

\begin{example}
\label{Ex Hausdorff}
Consider two curves $[(\theta, y)],[(\tilde \theta, \tilde y)] \in \mathcal F_{2,T} $, with canonical elements $(\theta, y), (\tilde \theta, \tilde y)$. Suppose that $(\theta(t), y(t))= (\tilde \theta(t), \tilde y(t))$ for $t\geq 2\pi $ and
\[
(\theta(t), y(t))= ( 0, \cos t, \sin t ) , \ \
(\tilde \theta(t), \tilde y(t))= ( 0, \cos (2\pi -t), \sin (2\pi- t) ) \quad \text{for } t \in [0,2\pi] .
\]
It is simple to check that the Hausdorff distance between $[(\theta, y)]$ and $[(\tilde \theta, \tilde y)]$ is zero but $d\left( [(\theta, y)],[(\tilde \theta, \tilde y)] \right) = d^+\left( [(\theta, y)],[(\tilde \theta, \tilde y)] \right) = 2$. $\square$
\end{example}

To illustrate the extended input-to-trajectory mapping, we return to Example~\ref{Ex noninvolutive system}:

\begin{example}
\label{Ex noninvolutive system c2}
Consider  Example \ref{Ex noninvolutive system}. By  Theorem \ref{T extension of input-to trajectory map}, the input-to-trajectory map is well defined. To compute the generalized trajectories corresponding to the generalized controls in the example, notice that here the system \eqref{Eq reduced system y} reduces to
\[
\dot y_1=w_1, \quad \dot y_2 = w_2, \quad \dot y_3 = y_2 w_2 .
\]
Let $(V_i,W_i), \ i=1,2,3$ be the parameterizations by length of the generalized controls $[(V_i,W_i)], \ i=1,2,3$ in Example \ref{Ex noninvolutive system c1}. Then,
\begin{align*}
&
(v_1,w_1)(t)= (0,1,0) \chi_{[0,1]}(t) + (0,0,1) \chi_{]1,2]}(t) + (1,0,0) \chi_{]2,+\infty[}(t),
\\ &
(v_2,w_2)(t)= \left(0,\frac{1}{\sqrt{2}},\frac{1}{\sqrt{2}}\right) \chi_{[0,\sqrt{2}]}(t) + (1,0,0) \chi_{]\sqrt{2},+\infty[}(t),
\\ &
(v_3,w_3)(t)= (0,0,1) \chi_{[0,1]}(t) + (0,1,0) \chi_{]1,2]}(t) + (1,0,0) \chi_{]2,+\infty[}(t) .
\end{align*}
Therefore,
\begin{align*}
&
(V_1,y_{(v_1,w_1)})(t) = (0,t,0,0) \chi_{[0,1]}(t) + (0,1,t-1,0) \chi_{]1,2]}(t) + (t-2,1,1,0) \chi_{]2,+\infty[}(t),
\\ &
(V_2,y_{(v_2,w_2)})(t) = \left(0,\frac{t}{\sqrt{2}},\frac{t}{\sqrt{2}}, \frac{t^2}{4}\right) \chi_{[0,\sqrt{2}]}(t) + \left(t-\sqrt{2},1,1,\frac 1 2 \right) \chi_{]\sqrt{2},+\infty[}(t) ,
\\ &
(V_3,y_{(v_3,w_3)})(t) = (0,0,t,0) \chi_{[0,1]}(t) + (0,t-1,1,t-1) \chi_{]1,2]}(t) + (t-2,1,1,1) \chi_{]2,+\infty[}(t)
\end{align*}
are parameterizations of the generalized trajectories corresponding to $[(V_i,W_i)], \ i=1,2,3$, respectively
(see Figure \ref{Fg generalized trajectories}).
\begin{figure}[h] \vspace{-0.3cm}
\begin{center}
\begin{minipage}[t]{12cm}
\includegraphics*[height=10cm]{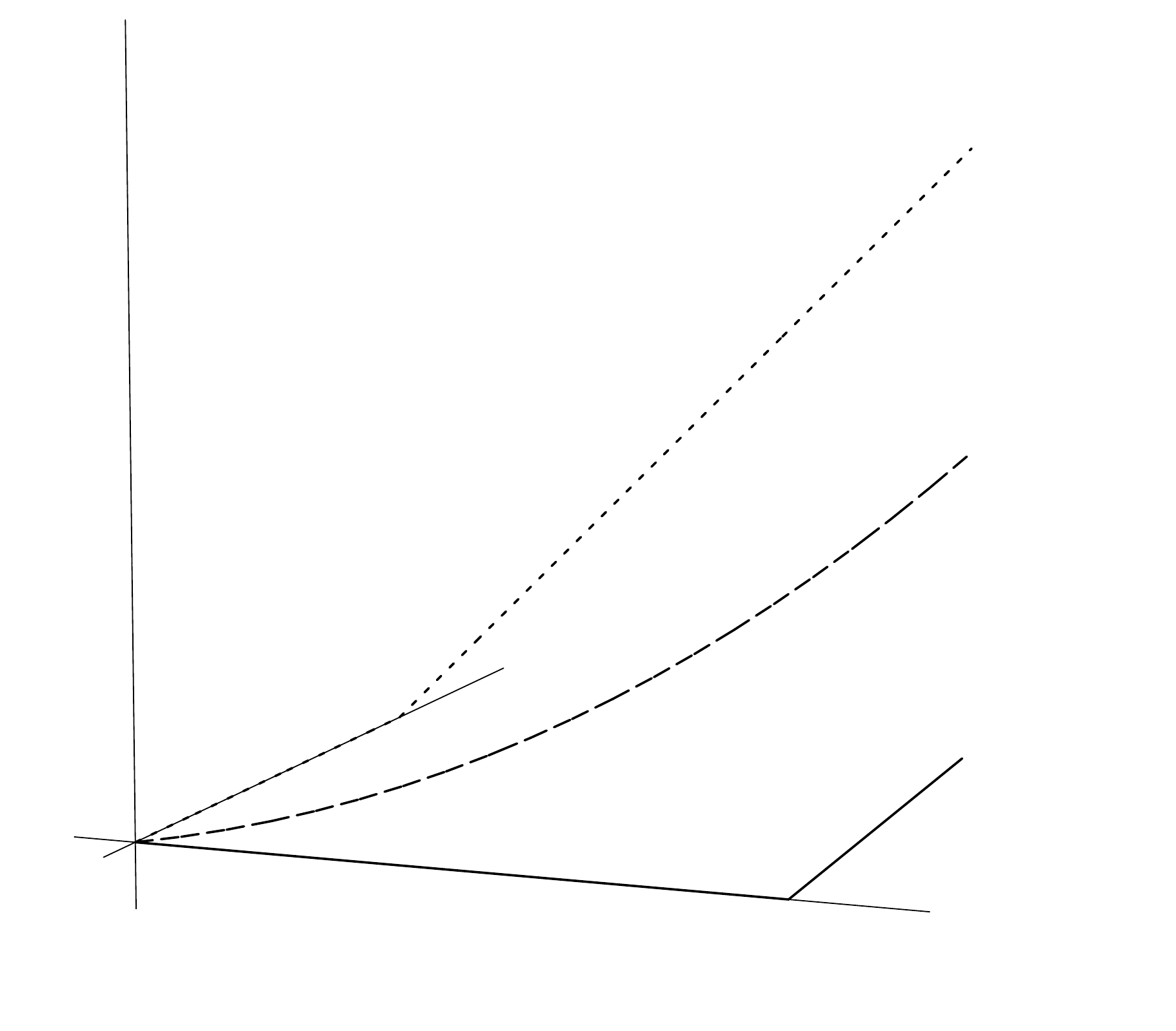}

\vspace{-1.2cm} \hspace{8.5cm} $x_1$

\vspace{-2.9cm} \hspace{4.8cm} $x_2$

\vspace{-6.5cm} \hspace{1.2cm} $x_3$

\end{minipage}
\end{center}
\vspace{-0.9cm}

\caption{
\label{Fg generalized trajectories}
The jumps of the generalized trajectories from Example \ref{Ex noninvolutive system c2}: $\left[ (V_1,y_{(v_1,w_1)}) \right]$ (solid line), $\left[ (V_2,y_{(v_2,w_2)}) \right]$ (dashed line), and $\left[ (V_3,y_{(v_3,w_3)}) \right]$ (doted line).
}
\end{figure}
Notice that
\begin{align*}
&
y_{(v_1,w_1)} \circ V_1^\#(t) =(1,1,0), \qquad
y_{(v_2,w_2)} \circ V_2^\#(t) =\left(1,1,\frac 1 2 \right),
\\ &
y_{(v_3,w_3)} \circ V_3^\#(t) =(1,1,1),
\end{align*}
for every $t >0$. $\square$
\end{example}

To compare our approach with the one developed by Miller and Rubinovich \cite{MiRu}, we formulate the following result, which  shows that every generalized trajectory of system \eqref{affine system}, as defined in \cite{MiRu} coincides with some $y_{(v,w)}\circ V^\#$, with $(V,W) \in \mathcal Y^0_{k,T}$.

\begin{proposition}
\label{P Miller trajectories}
Consider a sequence of ordinary controls $\left\{ u_i \in L_1^k[0,T] \right\}_{i \in \mathbb N} $, equibounded in $L_1[0,T]$-norm, such that the corresponding sequence of trajectories $\left\{ x_{u_i} \right\}_{i \in \mathbb N} $ is equibounded in $L_\infty[0,T]$-norm.

There is a subsequence $\{ u_{i_j} \}$ such that $\left[ ( t,U_{i_j}) \right] $ converges  towards some $[(V,W)] \in \mathcal F_{k,T}^0$, and $x_{u_{i_j}} (t) $ converges pointwise to $y_{(v,w)} \circ V^\#(t)$ in $[0,T]$, with the possible exception of a countable set of points. $\square$
\end{proposition}

\begin{proof}
The Proposition \ref{P compactness} guarantees existence of a convergent subsequence $\left\{ \left[ ( t,U_i) \right] \right\}_{i \in \mathbb N} $.
Thus, we can assume without loss of generality that $\left\{ \left[ ( t,U_i) \right] \right\}_{i \in \mathbb N} $ converges to some $[(V,W)] \in \mathcal F_{k,T}^0$. Let $(V,W)$ be the canonical parameterization of $[(V,W)]$, and $(V_i,W_i)$ be the canonical parameterization of $\left[ ( t,U_i) \right]$, for $i \in \mathbb N$.
Notice that $x_{u_i} = y_{(v_i,w_i)} \circ V_i^\# $ for every $i \in \mathbb N$.

Since the sequence $\left\{ x_{u_i} \right\}_{i \in \mathbb N} $ is equibounded, we can assume that the vector fields $f$, $g_1, \ldots , g_k$ have compact support.
In that case, the sequence $\left\{ (V_i,y_{(v_i,w_i)} ) \right\}$ is uniformly Lipschitz and converges uniformly towards $(V,y_{(v,w)})$ in the interval $\left[ 0, 1 + V^\#(T)\right]$.
Then, for any $t \in [0,T]$:
\begin{align*}
\left| y_i \circ V_i^\#(t) - y \circ V^\#(t) \right| \leq &
\left| y_i \circ V_i^\#(t) - y_i \circ V^\#(t) \right| +
\left| (y_i-y) \circ V_i^\#(t) \right| \leq
\\ \leq &
C \left| V_i^\#(t) -  V^\#(t) \right| +
\left\| y_i-y \right\|_{L_\infty\left[0,1 + V^\#(T)\right]} .
\end{align*}
Since
\[
\liminf_{i \rightarrow \infty} V_i^\# (t) \geq V^\#(t^-) , \quad
\limsup_{i \rightarrow \infty} V_i^\# (t) \leq V^\#(t^+),  \qquad \forall t \in [0,T],
\]
we see that $\lim\limits_{i \rightarrow \infty} x_{u_i}(t) = y \circ V^\#(t)$, with exceptions only at the points of discontinuity of $V^\#$. This  set is at most countable, and the result follows.
\end{proof}

\section{Auxiliary problem}
\label{S Problem reduction}

Coming back to the optimal control problem  \eqref{lag}--\eqref{boundary conditions} we show that parameterization by the arc length of the curves $[(t,U)]$ results in   its equivalent reformulation. Canonical parameterizations where introduced in \cite{BressanRampazzo88}, for Cauchy problems. Here we extend the analysis to Lagrange problems \eqref{lag}--\eqref{boundary conditions}. A similar reparameterization was introduced in \cite{MiRu}. \medskip

For each ordinary control $u \in L_1^k[0,1]$, the length function $\ell_{(t,U)}$ coincides with the function $\tau_u:[0,+\infty[ \mapsto [0,+\infty[$, given by
\begin{equation}
\label{Eq arc length function U}
\tau_u(t) = \int_0^t \sqrt{1+ |u|^2} \, ds,
\end{equation}
and the canonical parameterization of the curve $[(t,U)]$ is the function $(V,W)= \left( \tau_u^{-1}, U \circ \tau_u^{-1} \right) $.
The function $\tau_u$ admits an absolutely continuous inverse $\tau_u^{-1}$ with $\frac{d}{dt} \tau_u^{-1}= \frac{1}{\sqrt{1+|u|^2}}\circ \tau_u^{-1}$.

The reparameterized trajectory $x_u \circ \tau_u^{-1}$ coincides with $y_{(v,w)}$, the unique solution of \eqref{Eq reduced system y} with
\begin{equation}
\label{Eq transformation u to vw}
(v,w)= \left( \frac{1}{\sqrt{1+|u|^2}}\circ \tau_u^{-1}, \frac{u}{\sqrt{1+|u|^2}}\circ \tau_u^{-1} \right) .
\end{equation}
By the change of variable $t= \tau_u^{-1}$, the cost functional \eqref{lag} becomes
\begin{align*}
\int_0^1 \! L(x_u,u) dt =
\int_0^{\tau_u(1)} \!\!\! L(y_{(v,w)}, u \circ \tau_u^{-1}) \frac{1}{\sqrt{1+|u|^2}}\circ \tau_u^{-1} dt =
\int_0^{\tau_u(1)} \!\!\! L\left(y_{(v,w)}, \frac w v \right) v \, dt .
\end{align*}
Conversely, for any $T \in ]0,+\infty[$ and any measurable function $t \mapsto (v(t),w(t)) \in \mathbb R^{1+k}$ satisfying \begin{align*}
& v(t) >0, \ \ v(t)^2+|w(t)|^2 = 1 \quad \text{a.e. } t \in [0,T] ,
\qquad
\int_0^Tv(t) dt =1,
\end{align*}
there is a unique $u \in L_1^k[0,1]$ satisfying \eqref{Eq transformation u to vw}.

Therefore, the problem \eqref{lag}--\eqref{boundary conditions} is equivalent to
\begin{align}
\label{Eq reduced lagrangian}
& I(v,w,T)=\int_0^TL\left( y(t), \frac{w(t)}{v(t)}\right) v(t) dt \rightarrow \min , \\
\label{Eq reduced system}
& \dot \theta(t) = v(t) , \quad \dot y(t) = f(y(t))v(t) + G(y(t))w(t), \\
\label{Eq control constraints}
& v(t) >0, \quad v(t)^2+|w(t)|^2 =1 \qquad \text{a.e. } t \in [0,T], \\
\label{Eq boundary conditions}
& \theta(0)=0, \ \theta(T)=1, \ y(0)=x^0, \ y(T)=x^1,
\end{align}
with free $T\in ]0,+\infty[$.

It is crucial that the integrand in \eqref{Eq reduced lagrangian} is parametric in the terminology of L.C.Young \cite{Young}, i.e. invariant with respect to the dilation $(v,w) \to (\kappa v, \kappa w), \ \kappa \in \mathbb{R}_+$. This will later allow us to extend the functional \eqref{Eq reduced lagrangian} onto the class of Fr\'echet generalized controls.

The proof of the following Lemma is given in Appendix (Subsection \ref{SP L convexity of reduced lagrangean}).

\begin{lemma}
\label{L convexity of reduced lagrangean}

1. For every $y \in \mathbb R^n$, the function $(v,w)\mapsto L\left( y, \frac w v \right) v $ is convex in $]0,+\infty[\times \mathbb R^k$.

2. For every $(\hat y, \hat v, \hat w) \in \mathbb R^n \times [0,+\infty[\times \mathbb R^k $,
\begin{align*}
\liminf_{\scriptsize\begin{array}{c}
(y,v,w) \rightarrow (\hat y, \hat v, \hat w) \\
v>0
\end{array}} L\left( y, \frac w v \right) v = &
\liminf_{\scriptsize \begin{array}{c}
(v,w) \rightarrow (\hat v, \hat w) \\
v>0
\end{array}} L\left( \hat y, \frac w v \right) v =
\\ = &
\lim_{v \rightarrow \hat v, \ v >0 } L\left( \hat y, \frac {\hat w} v  \right) v . \ \square
\end{align*}
\end{lemma}

We define the new Lagrangian on $\mathbb{R}^n \times \mathbb{R}_+ \times \mathbb{R}^k$:
\begin{align}\label{lambda_extended}
& \lambda (y,v,w) =
\left\{ \begin{array}{ll}
L\left( y, \frac w v \right) v, & \text{for } v>0, \smallskip \\
\lim\limits_{\eta \rightarrow 0^+} L\left( y, \frac w \eta \right) \eta , & \text{for } v=0
\end{array} \right.
\end{align}

Lemma \ref{L convexity of reduced lagrangean} implies

\begin{corollary}
The function $\lambda (y,v,w)$, defined by \eqref{lambda_extended}
is the lower semicontinuous envelope of the function $(y,v,w) \mapsto L\left( y, \frac w v \right) v. \ \square$
\end{corollary}
\begin{remark}
Since the lower semicontinuous envelope of a convex function is convex, we conclude that $\lambda $ is convex with respect to $(v,w) \in [0,+\infty[\times \mathbb R^k$. $\square$
\end{remark}

Replacing the condition $v(t)>0 $ by $v(t)\geq 0 $ in \eqref{Eq control constraints}, one obtains a  problem with controls taking values in the compact set $\left\{ (v,w) \in \mathbb R^{1+k}: v\geq 0, v^2+|w|^2 = 1 \right\}$.
This is the so-called compactification technique, started probably in \cite{Gam}. For a recent contribution, see \cite{GuSar}.
By relaxing the control values to the convex hull $B_k^+$ of this set, we introduce the relaxed problem
\begin{align}
\label{Eq relaxed lagrangian}
& \widehat I(v,w,T)=\int_0^T \lambda \left( y(t), v(t), w(t) \right) dt \rightarrow \min , \\
\label{Eq relaxed system}
& \dot \theta(t) = v(t) , \quad \dot y(t) = f(y(t))v(t) + G(y(t))w(t), \\
\label{Eq control relaxed constraints}
& (v(t),w(t)) \in B_k^+=\{(v,w)| \ v \geq 0, \quad v^2+| w |^2 \leq 1\} \qquad \text{a.e. } t\in[0,T], \\
\label{Eq boundary conditions relaxed}
& \theta(0)=0, \ \theta(T)=1, \ y(0)=x^0, \ y(T)=x^1,
\end{align}
with free $T\in [0,+\infty[$.

\section{The cost functional for Fr\'echet generalized trajectories and controls}
\label{S the cost functional}

The argument used above shows that for an ordinary control $u(t)$, and the  graph of its primitive  $(t,U(t))$, one  gets   for each  $(V,W) \in [(t,U)]$:
\[
J(x_u,u) = \int_0^{V^{\#}(1)} \lambda\left( y_{(v,w)},v,w \right) d \tau .
\]
This suggests that the cost functional \eqref{lag} can be extended onto the space $\mathcal F_{k,1}^0$ and this extension should coincide with the functional
\begin{equation}
\label{Eq extended cost}
I([(V,W)])= \int_0^{V^{\#}(1)} \lambda\left( y_{(v,w)},v,w \right) d \tau  \qquad [(V,W)] \in \mathcal W_1 .
\end{equation}

First, note that the functional \eqref{Eq extended cost} is properly defined, that is:

\begin{proposition}
\label{P parameterization invariance of cost}
The mapping
\[
[(V,W)] \in \mathcal W_1 \mapsto \int_0^{V^{\#}(1)} \lambda\left( y_{(v,w)},v,w \right) d \tau
\]
does not depend on a particular representative  $(V,W) \in   [(V,W)].
\ \square$
\end{proposition}

\begin{proof}
Follow the argument of the proof of  Proposition \ref{P parameterization invariance of trajectories}.
\end{proof}

The following property of the functional \eqref{Eq extended cost} holds:
\begin{proposition}
\label{P cost lower semicontinuity}
The functional $I([(V,W)]) = \int_0^{V^{\#}(T)}\lambda (y_{(v,w)},v,w) dt $ is lower semicontinuous in the space of generalized controls $\left( \mathcal F_{k,T}^0,d_T^+ \right)$. $\square$
\end{proposition}

\begin{proof}
For any $0< \delta < \varepsilon < 1$ and $y \in \mathbb R^n$, $(v,w) \in B_k^+$ (satisfying the constraints \eqref{Eq control relaxed constraints}), we get
\begin{align*}
L\left( y, \frac{w}{v+\varepsilon} \right)(v+\varepsilon) = &
L\left( y, \frac{v+\delta}{v+\varepsilon} \frac{w}{v+\delta} + \frac{\varepsilon-\delta}{v+\varepsilon} 0\right)(v+\varepsilon) \leq
\\ \leq &
L\left( y, \frac{w}{v+\delta} \right)(v+\delta)
+L\left( y, 0 \right)(\varepsilon-\delta) .
\end{align*}
Passing to the limit at the right-hand side, as $\delta \to 0^+$, we invoke Lemma \ref{L convexity of reduced lagrangean}  to conclude
\[
\lambda (y,v,w) \geq
L\left( y, \frac{w}{v+\varepsilon} \right)(v+\varepsilon)
-L\left( y, 0 \right)\varepsilon ,
\]
and thus  $\lambda$ is bounded from below on compact sets.

Fix a sequence $\left\{ [(V_i,W_i)] \in \mathcal F_{k,T}^0 \right\}_{i \in \mathbb N} $ converging to $[(V,W)]$ with respect to the metric $d_T^+$, and let $\{ (V_i,W_i)\}_{i \in \mathbb N}$, $(V,W)$ be the respective canonical representatives.

By  Theorem \ref{T extension of input-to trajectory map}, $y_{(v_i,w_i)}$ converges uniformly to $y_{(v,w)}$.
By the Theorem \ref{T continuity Frechet to W}, $(v_i,w_i)$ converges to $(v,w)$ with respect to the $L_1$-norm. Therefore, there is a subsequence $(v_{i_j},w_{i_j})$ converging pointwise almost everywhere to $(v,w)$.

Using Fatou's Lemma:
\begin{align*}
&
\liminf_{j \rightarrow \infty} I\left( [(V_{i_j},W_{i_j})] \right) =
\liminf_{j \rightarrow \infty}  \int_0^{V_{i_j}^{\#}(T)}\lambda (y_{(v_{i_j},w_{i_j})},v_{i_j},w_{i_j}) dt \geq
\\ \geq &
\int_0^{V^{\#}(T)}\liminf_{j \rightarrow \infty} \lambda (y_{(v_{i_j},w_{i_j})},v_{i_j},w_{i_j}) dt \geq
\int_0^{V^{\#}(T)}\lambda (y_{(v,w)},v,w) dt = I\left( [(V,W)] \right) .
\end{align*}
\end{proof}

In order to extend the optimal control problem \eqref{lag}--\eqref{boundary conditions} onto the class of Fr\'echet generalized controls $\mathcal F_{k,1}^0$, let us recall some of previously obtained results:
\begin{itemize}
\item
The functional $[(V,W)] \mapsto I\left( [(V,W)] \right)$ is lower semicontinuous in $\mathcal F_{k,1}^0$ and
\[
I\left( [(t,U)] \right) = J(x_u,u) \qquad \forall u \in L_1^k[0,1] .
\]
\item
The input-to-trajectory map $[(V,W)] \mapsto [(V,y_{(v,w)})]$ is the unique continuous extension of the input-to-trajectory of system \eqref{affine system}. A generalized trajectory has a well defined endpoint given by $y_{(v,w)}\circ V^\#(1)$. This point does not depend on a particular $(V,y_{(v,w)}) \in [(V,y_{(v,w)})]$; it is the $x$-component of the point, at  which the Fr\'echet curve $[(V,y_{(v,w)})]$ crosses (leaves) the hyperplane $\{ (t,x)\in \mathbb R^{1+n}: t=1 \}$.
\end{itemize}
Basing on these considerations, we introduce the extended problem:
\begin{align}
\label{Eq generalized cost}
& I\left( [(V,W)] \right) \rightarrow \min ,
\\ &
\label{Eq generalized boundary conditions}
[(V,W)] \in \mathcal F_{k,1}^0, \quad
y_{(v,w)} \circ V^\#(T) = x^1 .
\end{align}
This problem is equivalent to the relaxed problem
\eqref{Eq relaxed lagrangian}--\eqref{Eq boundary conditions relaxed},
in the following sense:
\begin{itemize}
\item[{\bf i)}] If $(v,w)$ is an optimal control for the problem \eqref{Eq relaxed lagrangian}--\eqref{Eq boundary conditions relaxed}, then $[(V,W)]= \left[\int_0^{( \cdot )} (v,w) d \tau \right]$ is optimal for the problem \eqref{Eq generalized cost}--\eqref{Eq generalized boundary conditions} and the corresponding generalized trajectory is $\left[ ( V,y_{(v,w)})\right]$;
\item[{\bf ii)}] If  $[(V,W)]$ is optimal for the problem \eqref{Eq generalized cost}--\eqref{Eq generalized boundary conditions} and $(V,W)$ is its canonical representative, then $(v,w)=(\dot V, \dot W)$ is an optimal control for the problem \eqref{Eq relaxed lagrangian}--\eqref{Eq boundary conditions relaxed}.
\end{itemize}

\section{Existence of Fr\'echet generalized minimizers for integrands with linear growth }
\label{S existence}

Following the aforesaid, the problem
\eqref{lag}--\eqref{boundary conditions} admits a generalized solution if and only if the problem \eqref{Eq relaxed lagrangian}--\eqref{Eq boundary conditions relaxed} admits a solution.
In this Section we prove existence of minimizer for \eqref{Eq relaxed lagrangian}--\eqref{Eq boundary conditions relaxed}.

To this end we start with a classical Ascoli-Arzel\`a-Filippov argument  to obtain a general necessary and sufficient condition (Proposition \ref{P existence of solution}) for existence  of minimizer for the relaxed problem.  Then we prove that this condition is satisfied in the case of the integrand of linear growth in control variable(s).
\medskip

Consider the set $B_k^+$ defined by \eqref{Eq control relaxed constraints}
and let  for each $y \in \mathbb R^n$:
\begin{align*}
Q(y) = \Big\{
(\phi, v, f(y)v+G(y)w ) :  (v,w) \in B^+ , \phi \geq \lambda (y,v,w)
\Big\} .
\end{align*}
We pass to the differential inclusion form of the problem \eqref{Eq relaxed lagrangian}--\eqref{Eq boundary conditions relaxed}, following  the scheme developed in \cite{Cesa}:
\begin{align}
\label{Eq nonparametric lagrangian}
& C(T) \rightarrow \min, \\
\label{Eq nonparametric dynamics}
& \left( \dot C(t), \dot \theta(t), \dot y(t) \right) \in Q(y(t)) \quad \text{a.e. } t \in [0,T], \\
\label{Eq nonparametric boundary 1}
& C(0)=0, \ \ \theta (0) = 0, \ \ \theta(T) = 1, \\
\label{Eq nonparametric boundary 2}
& y(0)=x^0, \ \  y(T) = x^1,
\end{align}
with free $T\in[0,+\infty[$.

For each $y \in \mathbb R^n$, $\varepsilon >0 $, let $ Q(B_\varepsilon(y)) = \bigcup\limits_{|z-y|<\varepsilon} Q(z)$.
The following Lemma shows that the differential inclusion \eqref{Eq nonparametric dynamics} is continuous.

\begin{lemma}
\label{L continuity inclusion}
For every $y \in \mathbb R^n$, $\ Q(y) = \bigcap\limits_{\varepsilon >0} Q(B_\varepsilon(y))$. $\square$
\end{lemma}

\begin{proof}
Fix $(\phi,V) \in  \bigcap\limits_{\varepsilon >0}  Q(B_\varepsilon(y)) $. By definition, there is a sequence
\linebreak
$\left\{ (y_i,v_i,w_i) \in B_{\frac 1 i}(y) \times B^+ \right\}_{i \in \mathbb N}$ such that
\begin{align*}
&V = \left(v_i, f(y_i)v_i+G(y_i)w_i \right), \qquad 
\phi \geq \lambda (y_i,v_i,w_i)
\end{align*}
for every $i \in \mathbb N $.

Since $\{(v_i,w_i)\}$ is bounded, we can assume, passing to a subsequence, that it converges to some $(v,w) \in B^+$. By continuity of $f,G$, we have $V=\left(v,f(y)v+G(y)w \right)$. By  Lemma \ref{L convexity of reduced lagrangean}, $\lambda(y,v,w) \leq \liminf \lambda (y_i,v_i,w_i) \leq \phi$.
Therefore, $(\phi, V) \in Q(y)$.
\end{proof}

The classical Filippov selection theorem requires the Lagrangian to be continuous, a condition that is not guaranteed for the auxiliary Lagrangian $\lambda$. However we manage to  prove that Filippov's theorem holds for  $\lambda $.

\begin{proposition}
\label{P Filippov}
Fix $(C,\theta,y)$, a trajectory of the differential inclusion \eqref{Eq nonparametric dynamics}, defined in the interval $[0,T]$.

There is a measurable control $(v,w):[0,T] \mapsto B^+ $ such that
\begin{align*}
& \dot \theta(t) = v(t), \quad \dot y(t) = f(y(t))v(t) + G(y(t))w(t)  \\
& \dot C(t) \geq \lambda (y(t),v(t),w(t)) ,
\end{align*}
for a.e. $t \in [0,T]. \ \square$
\end{proposition}

\begin{proof}
See Appendix (Section \ref{SP P Filippov}).
\end{proof}

\begin{corollary}
\label{C equivalence nonparametric}
If $\{ (v_i,w_i)\}_{i \in \mathbb N}$ is a minimizing sequence for the problem \eqref{Eq relaxed lagrangian}--\eqref{Eq boundary conditions relaxed}, then
\begin{equation}
\label{Eq minimizing sequence nonparametric}
\left\{
\left(\int_0^{(\cdot)} \lambda (y_{(v_i,w_i},v_i,w_i) dt, V_i, y_{(v_i,w_i)} \right)
\right\}_{i \in \mathbb N}
\end{equation}
is a minimizing sequence for the  problem \eqref{Eq nonparametric lagrangian}--\eqref{Eq nonparametric boundary 2}.

The problem \eqref{Eq relaxed lagrangian}--\eqref{Eq boundary conditions relaxed} admits a solution if and only if the  problem \eqref{Eq nonparametric lagrangian}--\eqref{Eq nonparametric boundary 2} does. $\square$
\end{corollary}

\begin{proof}
Due to the Proposition \ref{P Filippov}, if  \eqref{Eq minimizing sequence nonparametric} fails to be a minimizing sequence for the problem \eqref{Eq nonparametric lagrangian}--\eqref{Eq nonparametric boundary 2}, then there would exist  an admissible control $(\hat v, \hat w)$, for which
  \[
\int_0^{\hat V^\#(1)} \lambda (y_{(\hat v, \hat w},\hat v,\hat w) dt <
\liminf \int_0^{V_i^\#(1)} \lambda (y_{(v_i,w_i},v_i,w_i) dt,
\]
which is a contradiction.

It follows that, whenever $(v,w)$ is a solution for the problem \eqref{Eq relaxed lagrangian}--\eqref{Eq boundary conditions relaxed}, then
\[
\left(\int_0^{(\cdot)} \lambda (y_{(v,w},v,w) dt, V, y(v,w) \right)
\]
must be  solution for the problem \eqref{Eq nonparametric lagrangian}--\eqref{Eq nonparametric boundary 2}.

Now, suppose that the  problem \eqref{Eq nonparametric lagrangian}--\eqref{Eq nonparametric boundary 2} admits a solution $(C,\theta,y)$. By  Proposition \ref{P Filippov}, there is an admissible control $(v,w)$ such that $I(v,w) \leq C(T)$. As far as  the infimum of problem \eqref{Eq relaxed lagrangian}--\eqref{Eq boundary conditions relaxed} cannot be strictly less than the minimum of \eqref{Eq nonparametric lagrangian}--\eqref{Eq nonparametric boundary 2}, it follows that $(v,w)$ must be a solution for \eqref{Eq relaxed lagrangian}--\eqref{Eq boundary conditions relaxed}.
\end{proof}

The following Proposition provides  necessary and sufficient condition for existence of solution of the problem \eqref{Eq nonparametric lagrangian}--\eqref{Eq nonparametric boundary 2}.

\begin{proposition}
\label{P existence of solution}
The problem \eqref{Eq nonparametric lagrangian}--\eqref{Eq nonparametric boundary 2} admits a solution if and only if it admits a minimizing sequence $\left\{ (C_i, \theta_i, y_i, T_i) \right\}_{i \in \mathbb N}$ for which  the sequences
\[
\{T_i \}_{i \in \mathbb N}, \quad
\left\{ \left\|(C_i, \theta_i, y_i) \right\|_{L_\infty[0,T_i]} \right\}_{i \in \mathbb N}
\]
are bounded. $\square$
\end{proposition}

\begin{proof}
The condition is clearly necessary.

To verify  sufficiency, note that such a sequence is bounded and equicontinuous. Therefore, the Ascoli-Arzel\`a Theorem guarantees existence of a subsequence, converging to a limit. Lemma \ref{L continuity inclusion} guarantees that the limit solves the differential inclusion \eqref{Eq nonparametric dynamics} and hence the limit is optimal.
\end{proof}

The Corollary \ref{C equivalence nonparametric} and the Proposition \ref{P existence of solution} immediately imply the following:

\begin{corollary}
\label{C existence of solution}
The relaxed problem \eqref{Eq relaxed lagrangian}--\eqref{Eq boundary conditions relaxed} admits a solution if and only if it has finite infimum and admits a minimizing sequence $\{ ( v_i, w_i )\} $ for which the sequences
\[
\left\{ T_i \right\}_{i \in \mathbb N}, \quad \left\{ \left\| y_{(v_i,w_i)} \right\|_{L_\infty[0,T_i]} \right\}_{i \in \mathbb N}
\]
are bounded. $\square$
\end{corollary}

Using this Corollary, we will prove existence of generalized solutions when the Lagrangian \eqref{lag} has linear growth with respect to controls:

\begin{proposition}
\label{P existence linear growth}
Suppose the following conditions hold:
\begin{itemize}
\item[i)]
There are constants $a \in \mathbb R$, $b>0$ such that
\[
L(x,u) \geq a+b|u| \qquad \forall (x,u) \in \mathbb R^{n+k} .
\]
\item[ii)]
There are constants $ \tilde a, \tilde b < +\infty $ such that
\begin{align*}
&
|G(x)u| \leq (\tilde a + \tilde b |x|)|u| + \tilde b L(x,u), \\
&|f(x)| \leq \tilde a+ \tilde b (|x|+ L(x,u))  \qquad \forall (x,u) \in \mathbb R^{n+k} .
\end{align*}
\end{itemize}
Then, the relaxed problem \eqref{Eq relaxed lagrangian}--\eqref{Eq boundary conditions relaxed} admits a minimizer, i.e., the original problem \eqref{lag}--\eqref{boundary conditions} admits a Fr\'echet generalized minimizer. $\square$
\end{proposition}

\begin{proof}
Adding a suitable constant to the Lagrangian $L$, we may replace the conditions (i), (ii) by
\begin{itemize}
\item[i$^\prime$)]
There is a constant $b>0$ such that
\[
L(x,u) \geq b(1+|u|) \qquad \forall (x,u) \in \mathbb R^{n+k} .
\]
\item[ii$^\prime$)]
There is a constant $ \tilde b < +\infty $ such that
\begin{align*}
&
|G(x)u| \leq \tilde b (|x-x^0||u| +  L(x,u)), \\
&|f(x)| \leq \tilde b (|x-x^0|+ L(x,u))  \qquad \forall (x,u) \in \mathbb R^{n+k} .
\end{align*}
\end{itemize}

Fix $\{ (v_i, w_i )\} $, a minimizing sequence for the problem \eqref{Eq relaxed lagrangian}--\eqref{Eq boundary conditions relaxed}.
Due to Propositions \ref{P parameterization invariance of trajectories} and \ref{P parameterization invariance of cost},
$
( v_i,  w_i )\circ \ell_{( V_i, W_i )}^\# $
is also a minimizing sequence.
Thus, we can assume that
\begin{equation}
\label{Eq control standartization}
v_i(t)^2 + |w_i(t)|^2 =1 \qquad \text{a.e. } t\geq 0, \ \forall i \in \mathbb N .
\end{equation}
In that case, the condition i$^\prime$) guarantees  that $\lambda (y,v_i,w_i) \geq \frac{b}{\sqrt{2}}\sqrt{v_i^2+|w_i|^2} = \frac{b}{\sqrt{2}}$. Hence $\hat I(v_i,w_i,T_i) \geq \frac{b}{\sqrt{2}} T_i$ and therefore the infimum of the problem is finite and the sequence $\{ T_i \}$ is bounded.

From the condition (ii$^\prime$), we get
\begin{align*}
&
|y_{(v_i,w_i)}(t)-x^0| \leq
\int_0^t | f(y_{(v_i,w_i)}) |v_i + |G(y_{(v_i,w_i)})w_i| d \tau \leq \\
\leq &
2 \int_0^t \tilde b |y_{(v_i,w_i)}-x^0| + \tilde b \lambda(y_{(v_i,w_i)},v_i,w_i)  d \tau
\leq
2\tilde b \hat I(v_i,w_i,T_i)+ 2\tilde b \int_0^t |y_{(v_i,w_i)}-x^0| d \tau ,
\end{align*}
and by  Gronwall's Lemma, the sequence $\{\|y_{(v_i,w_i)}\|_{L_\infty[0,T_i]} \}$  is bounded.
\end{proof}

\section{Lavrentiev gap for ordinary and generalized controls}
\label{S Lavrentiev phenomenon}

We briefly discuss, what we call Lavrentiev gap for the classes of ordinary and Fr\'echet generalized controls.

We say that the functional $I$ exhibits an $L_1^k[0,1]$-$\mathcal F_{k,1}^0$ {\it Lavrentiev(-type) gap}, if
\[
\inf\limits_{u \in L_1^k[0,1]} I\left( [(t,U)] \right) >
\inf\limits_{[(V,W)] \in \mathcal F_{k,1}^0} I \left( [(V,W)] \right) .
\]
This definition is complete only after we specify how to deal with the boundary conditions \eqref{boundary conditions}. One possibility is to consider approximations of generalized controls by ordinary controls that satisfy exactly the boundary conditions.
That is, to take the infima over the $u \in L_1^k[0,1]$ satisfying \eqref{boundary conditions} and over the $[(V,W)] \in \mathcal F_{k,1}^0$ satisfying \eqref{Eq generalized boundary conditions}.
In alternative, we may consider approximations of generalized controls by ordinary controls that satisfy \emph{approximately} the boundary conditions.

We adopt this last point of view, which leads to the
\begin{definition}\label{D Lavrentiev gap}
The functional $I$ exhibits an $L_1^k[0,1]$-$\mathcal F_{k,1}^0$ {\it Lavrentiev gap}, if
\[
\lim_{\varepsilon \rightarrow 0^+}
\inf_{\scriptsize \begin{array}{c}
u \in L_1^k[0,1] \\ |x_u (1)- x^1| \leq \varepsilon \end{array}} I\left( [(t,U)] \right) >
\inf_{\scriptsize \begin{array}{c}
[(V,W)] \in \mathcal F_{k,1}^0 \\ y_{(v,w)} \circ V^\#(1) = x^1 \end{array}} I \left( [(V,W)] \right) . \ \square
\]
\end{definition}

The original Lavrentiev phenomenon has been studied in the classical problem of the calculus of variations, where simple examples with a $W_{1,\infty}$-$W_{1,1}$ gap are known \cite{BallMizel,Mania}.
Some generalizations can be found in \cite{Sa97}.
Therefore, the occurrence of a $L_1^k[0,1]$-$\mathcal F_{k,1}^0$ gap is not surprising;  the following example shows that such gap is a real possibility

\begin{example}
\label{Ex Lavrentiev gap 1}
Consider the optimal control problem
\begin{align*}
&
J(u) = \int_0^1|x_1(t)|+ h\left(x_1(t),u(t) \right) dt \rightarrow \min, \\
& \dot x_1 = x_1+x_2, \quad \dot x_2 = u, \quad x(0) = (0,-1), \quad x(1) =(0,0) ,
\end{align*}
with
\[
h(x_1,u) = \left\{ \begin{array}{ll}
\max \left(|u| - \frac{1}{\sqrt{|x_1|}}, 0 \right)  & \text{for } x_1 \neq 0 ,
\smallskip \\
0  & \text{for } x_1 = 0 .
\end{array} \right.
\]
Note that the integrand $|x_1|+h(x_1,u)$ is a continuous function.
The problem is equivalent to
\begin{align}
&
I(v,w) = \int_0^T |y_1|v+h\left(y_1,\frac{w}{v}\right) v dt \rightarrow \min , \ T \ \mbox{- free},\label{Eq Ex Lavrentiev functional}
\\ &
\dot y_1 = (y_1+y_2)v, \quad \dot y_2 = w, \quad \dot{V}=v, \quad v\geq 0, \ v^2+w^2 = 1, \label{Eq Ex Lavrentiev constraints}
\\ &
y_1(0) = 0, \ y_2(0)=-1, \quad V(0)=0,  \quad V(T)=1, \quad y_1(T) = y_2(T)=0 .\label{Eq Ex Lavrentiev boundary cond}
\end{align}
The control $(\hat v, \hat w ) = (0,1) \chi_{[0,1]} + (1,0) \chi_{]1,+\infty[} $, $T=2$, satisfies the boundary condition and $I(\hat v, \hat w)=0$. Thus, it is optimal. It corresponds to a generalized control containing an impulse which is optimal for the initial problem.

We will show that for the problem \eqref{Eq Ex Lavrentiev functional}--\eqref{Eq Ex Lavrentiev boundary cond} there is a constant $C>0$ such that $I(v,w)\geq C$ whenever $v(t)>0$ almost everywhere,$V(T)=1$ and $|y_2(T)|$ is sufficiently small.  i.e. whenever a control in the original problem is ordinary and generates a trajectory with endpoint in a neighbourhood of the boundary condition $x(1)=(0,0)$.

Fix an arbitrary triple $(v,w,T)$ with $v(t)>0$ and $v(t)^2+w(t)^2 =1 $ for a.e. $t \geq 0$, such that $V(T)=1$ and $y_2(T)>-\frac 1 2$.

Let $ T_1= \min\left\{ t \in [0,T]: y_2(t) = -\frac 1 2 \right\}$, hence $-1 \leq y_2(t) \leq -1/2$ on $[0,T_1]$.  Given that $|\dot{y}_2(t)|=|w(t)|<1 $ we conclude $T_1>1/2$.

Then for $t \in [0,T_1]$:
\begin{equation}\label{Eq_Lavrentiev_dynam_y1}
y_1(t)=\int_0^t e^{\int_s^tv(\tau)d\tau}v(s)y_2(s)ds<0,
\end{equation}
and $ \dot{y}_1(t)=v(t)(y_1(t)+y_2(t))<0$. Hence
$
\dot y_1=v(t) (y_1(t)+y_2(t) )\leq v(t) y_2(t)$, and
\begin{align}
\label{Z013}
| y_1(t) | = &
-y_1(t) \geq
\int_0^t-y_2(t)v(t)dt \geq
\frac 1 2 \int_0^tv(t)dt=\frac 1 2 V(t) \qquad \forall t \in [0,T_1].
\end{align}
Then
\[
\int_0^T|y_1(t)|v(t)dt \geq
\int_0^{T_1} \frac 1 2 V(t)v(t)dt=\frac 1 4 (V(T_1))^2 ,
\]
and from
\eqref{Eq Ex Lavrentiev functional}
\[
I(v,w) \geq  \frac 1 4 (V(T_1))^2 + \int_0^{T_1}\left(|w(t)|-\sqrt{\frac{2}{V(t)}}v(t)\right)dt \geq \frac 1 4 (V(T_1))^2 +\frac 1 2 -\sqrt{\frac{V(T_1)}{2}};
\]
one notes that $\int_0^{T_1}|w(t)|dt \geq |y_2(T_1)-y_2(0)|=\frac 1 2$.

Given that $V(T_1) \in [0,1]$ we conclude that
\[
I(v,w) \geq \min_{z \in [0,1]}\left(\frac 1 4 z^2 +\frac 1 2 - \sqrt{\frac z 2}\right)=\frac 1 2 - \frac{3}{2^{8/3}} \geq 0.0275 . \ \square
\]
\end{example}

Now, we present some conditions that exclude a $L_1^k[0,1]$-$\mathcal F_{k,1}^0$ Lavrentiev gap.

\begin{proposition}
\label{P no gap continuous}
If the auxiliary Lagrangian $\lambda $ is continuous in  $\mathbb R^n \times B^+_k$, (see \eqref{Eq control relaxed constraints}), then the problem \eqref{lag}--\eqref{boundary conditions} does not exhibit Lavrentiev gap. $\square$
\end{proposition}

\begin{proof}
Pick a generalized control $[(V,W)]$ with canonical element $(V,W)$, and $T \in ]0,+\infty[$, satisfying the boundary condition $y_{(v,w)}(T) =x^1$, $V(T)=1$.

For each $\varepsilon >0$, let
\[
\left( V_\varepsilon(t), W_\varepsilon(t) \right) =
\left( V(\frac{t}{1+\varepsilon}) + \frac{\varepsilon t}{1 + \varepsilon }, W(\frac{t}{1+\varepsilon}) \right)
, \qquad t \geq 0 ,
\]
and let $T_\varepsilon $ be the unique $t$ solving $ V_\varepsilon(t) = 1$.
Then, $\left[ (V_\varepsilon, W_\varepsilon) \right] $ is an ordinary control and the Lebesgue's dominated convergence theorem guarantees that
\[
\lim_{\varepsilon \rightarrow 0^+} \int_0^{T_\varepsilon} \lambda \left( y_{(v_\varepsilon,w_\varepsilon)}, v_\varepsilon, w_\varepsilon \right) dt =
\int_0^T \lambda \left( y_{(v,w)}, v, w \right) dt .
\]
\end{proof}

The following example shows that there are problems in which the Fr\'echet generalized minimizer contain jumps along discontinuities of the auxiliary Lagrangian and yet have no Lavrentiev gap. Thus, continuity of the auxiliary Lagrangian is \emph{not} a necessary condition to exclude existence of gap.

\begin{example}
\label{Ex Lavrentiev gap 2}
Consider the optimal control problem
\begin{align*}
&
J(u) = \int_0^1|x_1(t)|^\alpha u(t)^2 dt \rightarrow \min, \\
& \dot x_1 = x_1+x_2, \quad \dot x_2 = u, \quad x(0) = (0,-1), \quad x(1) =(0,0) ,
\end{align*}
with $\alpha >0$ constant.

The auxiliary Lagrangian is
\[
\lambda(y_1,y_2,v,w) = \lambda(y_1,v,w) = \left\{\begin{array}{ll}
|y_1|^\alpha \frac{w^2}{v}, & \text{if } v \neq 0, \\
0 , & \text{if } w=0 \ \text{or } y_1=0,\\
+ \infty , & \text{if } v=0, \ y_1 \neq 0, \ w \neq 0 .
\end{array} \right.
\]
Clearly, it is discontinuous at the points $(0,0,w)$, $w \in \mathbb R$, for every positive $\alpha$.
The auxiliary problem is
\begin{align*}
&
I(v,w) = \int_0^{T} \lambda (y_1,v,w) dt \rightarrow \min ,
\\ &
\dot y_1 = (y_1+y_2)v, \quad \dot y_2 = w, \quad v\geq 0, \ v^2+w^2 = 1,
\\ &
y(0) = (0,-1), \quad V(T)=1, \quad y(T) = (0,0) .
\end{align*}
The control $(\hat v, \hat w ) = (0,1) \chi_{[0,1]} + (1,0) \chi_{]1,+\infty[} $ satisfies the boundary condition with $T=2$ and $I(\hat v, \hat w)=0$. Thus, it is optimal. It corresponds to a Fr\'echet generalized control with an impulse at $t=0$.

Now, consider the approximation of the generalized minimizer by ordinary controls corresponding to $(v_\eta, \hat w ) = (\eta,1) \chi_{[0,1]} + (1,0) \chi_{]1,+\infty[} $.
A simple computation shows that $V_\eta(2-\eta)=1$, $y_{(v_\eta, \hat w)}(2-\eta) = O(\eta)$ and $I(v_\eta, \hat w) = O(\eta^{\alpha -1})$.
Thus, the problem has no Lavrentiev gap when $\alpha >1$.

The argument breaks down for $\alpha \leq 1$. Indeed, it can be shown that the problem has a Lavrentiev gap when $\alpha \in ]0,1[$. $\square$
\end{example}

The Proposition \ref{P no gap continuous} has the following immediate corollary.

\begin{corollary}
\label{C no gap homogeneous}
Suppose that the Lagrangian can be written as
\[
L(x,u) = L_1(x)+L_2(x,u), \qquad \forall (x,u) \in \mathbb R^{n+k} ,
\]
with $u \mapsto L_2(x,u)$ positively homogeneous of degree $1$ for every $x \in \mathbb R^n$.
Then, the problem \eqref{lag}--\eqref{boundary conditions} has no Lavrentiev gap in the sense of Definition \ref{D Lavrentiev gap}. $\square$
\end{corollary}

\begin{proof}
If the assumption holds, then
$
\lambda(y,v,w) = L_1(y)v+L_2(y,w)
$.
\end{proof}

\begin{remark}
\label{Rmk 3}
Linear growth of the Lagrangian with respect to control does not guarantee lack of Lavrentiev gap.

To see this, consider the same dynamics and boundary conditions as in Example \ref{Ex Lavrentiev gap 1}, and introduce the modified functional
\[
\tilde J(u) = \int_0^1|x_1(t)|+ h\left(x_1(t),u(t) \right) + \varepsilon |u(t)| dt,
\]
with $\varepsilon>0$, a small constant.
Existence of generalized minimizer is guaranteed by Proposition \ref{P existence linear growth}.

The inequality $\tilde J(u) \geq C$ holds for every ordinary control satisfying the boundary condition. However, for the generalized minimizer given in Example \ref{Ex Lavrentiev gap 1}, we have $\tilde I(\hat v, \hat w) = \varepsilon $, and therefore $\inf\limits_{[(v,w)] \in \mathcal F_{1,1}^0} \tilde I \left( [(v,w)] \right) < \inf\limits_{u \in L_1[0,1]} \tilde J(u)$ for sufficiently small $\varepsilon>0$.
$\square$
\end{remark}

We conclude this section with two further cases where Lavrentiev gap cannot occur.

\begin{proposition}
\label{Rm no Lavrenteev}
If $L(x,u) = L_1(x) + L_2(u)$ with $L_1$ continuous and $L_2$ convex, then the optimal control problem does not have a Lavrenteev gap. $\square$
\end{proposition}

\begin{proof}
Since $y_{(v+\eta, w)} \rightarrow y_{(v,w)}$ uniformly when $\eta \rightarrow 0^+$, and
\begin{align*}
&\left( L_1(y_{(v+\eta,w)}) - L_2\left( \frac w{v+\eta} \right) \right) (v+\eta ) \leq
\\ \leq &
\lambda (y_{(v,w)},v,w) + \left( L_1(y_{(v+\eta,w)}) - L_1(y_{(v,w)}) \right) (v+\eta ) + L_2(0) \eta,
\end{align*}
The result follows from Lebesgue's dominated convergence theorem.
\end{proof}

\begin{proposition}
\label{P no gap no drift}
If $f\equiv 0$ (i.e., the system \eqref{affine system} has no drift), then the problem \eqref{lag}--\eqref{boundary conditions} has no Lavrentiev gap in the sense of Definition \ref{D Lavrentiev gap}. $\square$
\end{proposition}

\begin{proof}
If the system \eqref{affine system} has no drift, then
$y_{(v,w)}=y_{(\tilde v, w)}$ for every $v, \tilde v, w$.
Since $L(y, \frac{w}{v+\varepsilon})(v+\varepsilon) \leq \lambda(y,v,w) + \varepsilon L(y,0)$, it follows that
\[
\lim_{\varepsilon \rightarrow 0^+} I(V+\varepsilon t,W) = I(V,W)
\]
for every generalized control $[(V,W)]$.
\end{proof}

\section{Example}
\label{S Example}

We provide an example of a Lagrange variational problem with a functional of linear growth, for which the minimum is attained at a generalized minimizer.

The set of all horizontal curves in the Heisenberg group can be identified with the set of trajectories of the control system
\begin{align}
\label{Eq system Heisenberg}
\dot x_1 = u_1, \quad \dot x_2 = u_2, \quad \dot x_3 =2x_2u_1-2x_1u_2 .
\end{align}
By adding a smooth drift $f$, one obtains a control-affine system
\begin{align}
\label{Eq system Heisenberg affine}
\left( \begin{array}{c}
\dot x_1 \\ \dot x_2 \\ \dot x_3
\end{array} \right) =
\left( \begin{array}{c}
f_1(x) \\ f_2(x) \\ f_3(x)
\end{array} \right) +
\left( \begin{array}{c}
1 \\ 0 \\ 2x_2
\end{array} \right) u_1 +
\left( \begin{array}{c}
0 \\ 1 \\ -2x_1
\end{array} \right) u_2
\end{align}
We wish to minimize the functional
\begin{align}
\label{Eq cost example}
J(u_1,u_2) = \int_0^1 \sqrt{1+u_1^2+u_2^2} \ dt,
\end{align}
under the  boundary conditions $x(0)=\overline{x}$, $x(1)=\overline{\overline{x} }$.

This problem satisfies the assumptions of Proposition \ref{P existence linear growth}, provided $f$ does not have supralinear growth with respect to $x$. Therefore, it has a generalized solution in the class $\mathcal F_{2,1}^0$.

The auxiliary Lagrangian is
\[
\lambda (y,v,w) = v\sqrt{1+\left( \frac{w_1}{v} \right)^2 +\left( \frac{w_2}{v} \right)^2} =
\sqrt{v^2+w_1^2+w_2^2} .
\]
Therefore, the Proposition \ref{P no gap continuous} guarantees that the problem \eqref{Eq system Heisenberg affine}--\eqref{Eq cost example} does not have a $L_1^2[0,1]$-$\mathcal F_{2,1}^0$ Lavrentiev gap.

The extension of the problem \eqref{Eq system Heisenberg affine}--\eqref{Eq cost example} is equivalent to the problem
\begin{align}
\label{Eq example reduced Lagrangian}
& T \rightarrow \min ,
\\ &
\label{Eq example reduced dynamics}
\left( \begin{array}{c}
\dot y_1 \\ \dot y_2 \\ \dot y_3
\end{array} \right) =
\left( \begin{array}{c}
f_1(y) \\ f_2(y) \\ f_3(y)
\end{array} \right) v +
\left( \begin{array}{c}
1 \\ 0 \\ 2y_2
\end{array} \right) w_1 +
\left( \begin{array}{c}
0 \\ 1 \\ -2y_1
\end{array} \right) w_2 , \quad \dot{V}=v,
\\ &
\label{Eq example control values}
v \geq 0, \quad v^2+w_1^2+w_2^2 =1,
\\ &
\label{Eq example boundary conditions}
y(0) = \overline{x}, \quad V(0)=0, \quad V(T)=1, \quad y (T) = \overline{ \overline{x} } .
\end{align}
Optimal controls for this problem satisfy the Pontryagin maximum principle with Hamiltonian
\begin{align*}
H=&
\left( \lambda_1 f_1 (y)+\lambda_2 f_2(y)+\lambda_3 f_3(y) + \lambda_4 \right)v +
\left( \lambda_1+\lambda_3y_2\right) w_1 +
\left( \lambda_2-\lambda_3y_1\right) w_2 .
\end{align*}
An optimal trajectory of the problem \eqref{Eq system Heisenberg affine}--\eqref{Eq cost example} exhibits a jump if there is an interval where the corresponding extremal of the problem \eqref{Eq example reduced Lagrangian}--\eqref{Eq example boundary conditions} satisfies
\[
\lambda_1 f_1 (y)+\lambda_2 f_2(y)+\lambda_3 f_3(y) + \lambda_4 \leq 0 ,
\]
and hence, by the Pontryagin maximum principle, $v(t) =0$.
Jump paths are sub-Riemannian geodesics of the Heisenberg group.

The presence or absence of jumps in optimal solutions depends on the drift vector field $f$.

\subsection{Constant drift} For example, if the drift is a constant vector field of the form $f \equiv (0,0,C)$, one can easily conclude, that all optimal trajectories are continuous.

Indeed in this case the Hamiltonian amounts  to
\[
H=
\left( \lambda_3 C + \lambda_4 \right)v +
\left( \lambda_1+\lambda_3y_2\right) w_1 +
\left( \lambda_2-\lambda_3y_1\right) w_2 ,
\]
and the extremals satisfy $\dot \lambda_3= \dot \lambda_4 \equiv 0,\ v(t)=\max\{0,\lambda_3C+\lambda_4 \}$. It follows that $\lambda_3C+\lambda_4$ is  constant and,  if the constant is positive, then  $v(\cdot)$ does not vanish and  the extremal trajectory is continuous,  or, if it is non-positive, then $v(\cdot)$ vanishes identically. The latter possibility is incompatible with the condition $\int_0^T v(s)ds=1$.

\subsection{Case of linear drift}
Contrasting with the case above, for the linear drift vector field  $f(x)=(0,0,-x_3)$ optimal trajectories may have a jump. This happens, for example, for the boundary conditions $\overline{x}=0$, $\overline{\overline{x}}=(0,0,C)$, whenever  $C>0$ is large enough. We prove that in this case the optimal trajectory consists of an analytic arc in the interval $[0,1[$ and a final jump at $t=1$.

Let us write the equations of Pontryagin Maximum Principle with the Hamiltonian
\[
H=
\left( \lambda_4 - \lambda_3y_3 \right)v +
\left( \lambda_1+2\lambda_3y_2\right) w_1 +
\left( \lambda_2-2\lambda_3y_1\right) w_2 .
\]
The adjoint vector satisfies the system
\begin{equation}
\label{Eq dynamic adjoint vector}
\dot \lambda_1 = 2 \lambda_3 w_2, \quad
\dot \lambda_2 = -2 \lambda_3 w_1, \quad
\dot \lambda_3 = \lambda_3 v, \quad
\dot \lambda_4 = 0.
\end{equation}
As far as Hamiltonian $H$ is homogeneous, we  may consider the abnormal case  $H\equiv 0$ and the normal one: $H\equiv 1$.

\begin{lemma}
Abnormal extremals for this problem are trivial: $y(t) \equiv 0. \ \square$
\end{lemma}

\begin{proof}
The identity $H \equiv 0$ implies
\begin{equation}
\label{Z016}
\lambda_1+2\lambda_3y_2 \equiv 0 , \quad
\lambda_2-2\lambda_3 y_1 \equiv 0, \quad
\lambda_4-\lambda_3y_3 \leq 0 .
\end{equation}
Differentiating the first two equalities, one obtains
\[
\lambda_3(y_2v+2w_2) = \lambda_3(y_1v+2w_1)=0 .
\]
Besides
\begin{equation}
\label{Eq lambda3}
\lambda_3(t)=e^{V(t)}\lambda_3(0) ,
\end{equation}
and if $\lambda_3(t)$ vanishes at a point, then $ \lambda_3\equiv 0 $  and by
 \eqref{Z016} $\lambda_1 \equiv \lambda_2 \equiv 0, \ \lambda_4 <0$ and $v \equiv 0$, meaning that the end-point condition $V(T)=1$ can not be achieved.

If    $y_2v+2w_2 = y_1v+2w_1 \equiv 0 $, then  $1=v^2+w_1^2+w_2^2=v^2(1+y_1^2/4+y_2^2/4) $ and
$v= \frac{2}{\sqrt{4+y_1^2+y_2^2}}$ is absolutely continuous with
\[
\frac{dv}{dt}=
2 \frac{y_1^2+y_2^2}{\left( 4+y_1^2+y_2^2 \right)^2} =
\frac 1 2 v^2(1-v^2) .
\]
Besides $v(0)= \frac{2}{\sqrt{4+(y_1(0))^2+(y_2(0))^2}}=1$, and hence $v(t)\equiv 1, w_1(t)=w_2(t) \equiv 0$, which results in a trivial trajectory $y \equiv 0$
\end{proof}

Now, consider an extremal $(y,\lambda)=(y_1,y_2,y_3,\lambda_1,\lambda_2,\lambda_3,\lambda_4)$ such that $H\equiv 1$, $y(0)=0$ and $y(T)=(0,0, C)$,  $C>1$.
The extremal controls are
\begin{equation}
\label{Eq example optimal controls}
v=\max (0,\lambda_4-\lambda_3 y_3), \qquad
w_1= \lambda_1+2\lambda_3y_2, \qquad
w_2= \lambda_2-2\lambda_3y_1 .
\end{equation}

We will use the following three lemmata.

\begin{lemma}
\label{L decreasing v}
For the imposed boundary conditions, the extremal control $v(\cdot)$ is monotonously decreasing and $0<\lambda_4 \leq 1$. $\square$
\end{lemma}

\begin{proof}
Let $\sigma_v(t)=\lambda_4-\lambda_3(t) y_3(t)$ be the switching function, which determines $v(t)$ along the extremal.
From \eqref{Eq example optimal controls}  and the dynamics \eqref{Eq example reduced dynamics}, \eqref{Eq dynamic adjoint vector}, we have
\begin{align*}
\frac{d}{dt}\left( \lambda_1y_2-\lambda_2y_1\right) =
2 \lambda_3w_2y_1 +\lambda_1w_2 +2\lambda_3w_1y_1-\lambda_2w_1 =
w_1w_2-w_2w_1=0 ,
\end{align*}
and therefore
\begin{equation}
\label{Eq conserved quantity}
\lambda_1y_2-\lambda_2y_1 \equiv 0
\end{equation}
along any extremal trajectory with $y(0)=0$.
Further,
\begin{align*}
\frac{d}{dt}(\lambda_3y_3) = &
\lambda_3vy_3 + \lambda_3(-y_3v+2y_2w_1-2y_1w_2) =
\\ = &
2\lambda_3(\lambda_1y_2-\lambda_2y_1)+4\lambda_3^2(y_1^2+y_2^2) =
4\lambda_3^2(y_1^2+y_2^2) \geq 0.
\end{align*}
Therefore,
$\frac{d}{dt}(\lambda_4 - \lambda_3 y_3) = -\frac{d}{dt}(\lambda_3y_3) \leq 0$ and hence the extremal control $v$ is a monotonically decreasing function. Since the boundary condition \eqref{Eq example boundary conditions} requires $v$ to be positive in some interval, we see that $\lambda_4>0$. The equality $H \equiv 1$ implies $\lambda_4 = \lambda_4- \lambda_3(0) y_3(0) \leq 1$.
\end{proof}

\begin{lemma}
\label{L rescaling}
The trajectories of the system
\[
\dot y_1=w_1, \qquad
\dot y_2=w_2, \qquad
\dot y_3=-y_3v+2y_2w_1-2y_1w_2, \qquad
y(0)=0 ,
\]
 are invariant with respect to the dilation
 $$(v,w_1,w_2,y_1,y_2,y_3) \to (v,\eta w_1, \eta w_2, \eta y_1,\eta y_2,\eta^2 y_3). \ \square $$
 \end{lemma}

\begin{proof}
A direct verification.
\end{proof}

\begin{lemma} \label{L attainable set bounded}
For each $T \geq 0$ the attainable set of the system  \eqref{Eq example reduced dynamics} is bounded; in particular the time $T_C$, needed to attain the point $(0,0,C)$, grows to $+\infty$ as $C \to +\infty . \ \square$.
\end{lemma}

\begin{proof}
Notice that the right-hand side of \eqref{Eq example reduced dynamics} is bounded by a linear function of $|y|$.
\end{proof}

Now, let $(\hat y,\hat \lambda)$ be an extremal satisfying the boundary conditions \eqref{Eq example boundary conditions}, and let
$(\hat v,\hat w) =(\hat v,\hat w_1,\hat w_2) $ be the corresponding extremal control.
Assume the extremal value of the functional to be  $\hat T$ and   $\hat v(t)>0$ on  $[0, \hat T[$.
We will  prove that if  $C>0$  is large enough, the  extremal cannot be optimal.

We proceed by showing that there is a control $(v,w)$ with $v\geq 0$ such that the corresponding trajectory of \eqref{Eq example reduced dynamics} satisfies $V(\hat T)=1$, $y(\hat T) =(0,0,C) $, and
\[
\int_0^{\hat T} \sqrt{v^2+w_1^2+w_2^2} dt < \hat T .
\]
This inequality requires that $v^2+w_1^2+w_2^2\not \equiv 1$, but the Propositions \ref{P parameterization invariance of trajectories} and \ref{P parameterization invariance of cost} show that $(v,w)$ can be transformed by a time reparameterization into a control satisfying \eqref{Eq example control values} and \eqref{Eq example boundary conditions} for  $T = \int_0^{\hat T} \sqrt{v^2+w_1^2+w_2^2} dt$ and therefore $\hat T$ is not minimal.

For each $\varepsilon \in ]0 , \hat T [$, let $a= \int_{\hat T-\varepsilon}^{\hat T} \hat v dt $. We avoid the notation $a(\varepsilon)$, but keep the dependence on $\varepsilon$ in mind. In particular, $0 < a \leq  \varepsilon$.

Fix $\varepsilon$ and consider the modified control $\tilde v$, defined as
\[
\tilde v(t) =
\left\{ \begin{array}{ll}
\hat v(t)+1, & t \in [0,a], \smallskip \\
\hat v(t), & t \in [a,\hat T- \varepsilon [, \smallskip
\\
0, & t \in [\hat T- \varepsilon, \hat T],
\end{array} \right.
\]
and let $\tilde y$ be the trajectory of the system \eqref{Eq example reduced dynamics} for the control $(\tilde v, \hat w)$.

Then, $\tilde V(\hat T) = \hat V(\hat T) =1$, $\tilde y_1 \equiv y_1$, and $\tilde y_2\equiv y_2$. Further, for any $t \geq 0$:
\begin{align*}
\tilde y_3(t) - \hat y(t) = &
\int_0^t \hat y_3 v - \tilde y_3 \tilde v d \tau =
\int_0^t -(\tilde y_3 - \hat y_3) \tilde v + \hat y_3 (\hat v - \tilde v ) d \tau .
\end{align*}
It follows that
\begin{align*}
\tilde y_3(\hat T) - \hat y(\hat T) =
e^{-1}\int_0^{\hat T} e^{\tilde V(\tau)}\hat y_3(\hat v - \tilde v) d \tau =
- \int_0^a e^{\tilde V(\tau)-1}\hat y_3 d \tau +
\int_{\hat T - \varepsilon}^{\hat T} \hat y_3 \hat v d \tau .
\end{align*}
Since $|\hat y_3(t)| \leq 2t^2$ and $\lim\limits_{t \rightarrow \hat T} \hat y_3(t) =C$, there is a constant $k \in ]0,+\infty[$ such that
\begin{align}
\label{Z022}
\tilde y_3(\hat T) > C(1+a-k\varepsilon a) ,
\end{align}
for every $C \in ]0,+\infty[$ and every sufficiently small $\varepsilon >0$.

Let $\eta = \sqrt{\frac{C}{\tilde y_3(\hat T)}}$.
Due to Lemma \ref{L rescaling}, the control $(v,w) = ( \tilde v, \eta \hat w_1, \eta \hat w_2) $  satisfies the boundary conditions \eqref{Eq example boundary conditions}, and we estimate the functional
\begin{align*}
&
\int_0^{\hat T} \sqrt{v^2+w_1^2+w_2^2} dt =
\\ = &
\int_0^a \sqrt{(1+\hat v)^2 + \eta^2 (1- \hat v^2)} dt +
\int_a^{\hat T-\varepsilon} \sqrt{ \hat v ^2 + \eta^2 (1- \hat v^2)} dt +
\int_{\hat T-\varepsilon}^{\hat T} \sqrt{ \eta^2 (1- \hat v^2)} dt \leq
\\ \leq &
\int_0^{\hat T} \sqrt{ \hat v ^2 + \eta^2 (1- \hat v^2)} dt +
\int_0^a \sqrt{ 1+2\hat v + \hat v^2 + \eta^2 (1- \hat v^2)} -  \sqrt{ \hat v ^2 + \eta^2 (1- \hat v^2)} dt .
\end{align*}
Since $1+\hat v \leq 3$, the second integral is bounded by $\sqrt 3 a$ and therefore
\begin{align*}
&
\int_0^{\hat T} \sqrt{v^2+w_1^2+w_2^2} dt \leq
\int_0^{\hat T} \sqrt{1-(1-\eta^2)(1-\hat v^2)} dt + \sqrt 3 a \leq
\\ \leq &
\int_0^{\hat T} 1-\frac{1-\eta^2}2(1-\hat v^2) dt + \sqrt 3 a \leq
\hat T - \frac{1-\eta^2}2\int_0^{\hat T} 1 - \hat v dt + \sqrt 3 a =
\\ = &
\hat T - \frac{1-\eta^2}2 (\hat T- 1) + \sqrt 3 a .
\end{align*}
Since \eqref{Z022} implies
$ 1- \eta^2 > \frac{1-k\varepsilon}{1+(1-k\varepsilon)a} a $,
the estimate above yields
\begin{align*}
&
\int_0^{\hat T} \sqrt{v^2+w_1^2+w_2^2} dt <
\hat T - \left( \frac{\hat T -1}2 \frac{1-k\varepsilon}{1+(1-k\varepsilon)a} - \sqrt 3 \right) a < \hat T ,
\end{align*}
provided $\varepsilon >0$ is sufficiently small and $\hat T > 1 +2 \sqrt 3$.
Due to Lemma \ref{L attainable set bounded}, this last condition holds for every sufficiently large $C>0$. For such $C$ no extremal satisfying $\hat v >0$ in $[0,\hat T[ $ can be optimal.

\section{Appendix: proofs of technical results}
\label{S appendix}

\subsection{Proof of Lemma~\ref{L orientation}}
\label{SP L orientation}

\begin{proof}
Suppose that \eqref{Eq orientation} holds and pick a sequence $\{ \beta_i \in \mathcal T \}_{i \in \mathbb N}$ such that $$\lim\limits_{i \rightarrow \infty} \left\| g_1-g_2 \circ \beta _i \right\|_{L_\infty[0,+\infty [} = 0.$$

For each $i \in \mathbb N$, let $\alpha_{1,i}$ denote the inverse function of $t \mapsto t + \beta_i(t)$, and let $\alpha_{2,i}= \beta_i \circ \alpha_{1,i} $.
Since
\[
\dot \alpha_{1,i} = \frac{1}{1+\dot \beta_i \circ \alpha_{1,i} } , \qquad
\dot \alpha_{2,i} = \frac{\dot \beta_i \circ \alpha_{1,i}}{1+\dot \beta_i \circ \alpha_{1,i} },
\]
the sequence $(\alpha_{1,i}, \alpha_{2,i} )$ is uniformly bounded and equicontinuous in compact intervals. Due to the Ascoli-Arzel\`a theorem, it admits a subsequence converging uniformly in compact intervals towards some absolutely continuous nondecreasing functions $(\alpha_1, \alpha_2)$. Due to continuity of $g_1,g_2$, $(\alpha_1, \alpha_2)$ satisfy {\bf (a)}.

Since $\alpha_{1,i}+\alpha_{2,i}=Id$, it follows that $\alpha_1+\alpha_2=Id$ and therefore {\bf (b)} holds.

Suppose that $ \alpha_1(\infty ) = T < +\infty $. Due to continuity of $g_1$, $g_1(T^-)= g_1(T)$. For any $t>T$, and any $i \in \mathbb N$:
\begin{align*}
&
\left| g_1(T)- g_1(t) \right| =
\left| g_1(T)- g_1\circ \alpha_{1,i}\left( \alpha_{1,i}^{-1}(t)\right) \right| \leq
\\ \leq &
\left| g_1(T)- g_2\circ \alpha_{2,i}\left( \alpha_{1,i}^{-1}(t)\right) \right| +
\left| g_2\circ \alpha_{2,i}\left( \alpha_{1,i}^{-1}(t)\right)- g_1\circ \alpha_{1,i}\left( \alpha_{1,i}^{-1}(t)\right) \right| \leq
\\ \leq &
\left| g_1(T)- g_2\circ \alpha_{2,i}\left( \alpha_{1,i}^{-1}(t)\right) \right| +
\left\| g_2\circ \alpha_{2,i} - g_1\circ \alpha_{1,i} \right\|_{L_\infty[0,+\infty[} .
\end{align*}
By assumption, $\lim\limits_{i \rightarrow \infty} \alpha_{1,i}^{-1}(t) = + \infty$ and therefore $\lim\limits_{i \rightarrow \infty} \alpha_{2,i} \left( \alpha_{1,i}^{-1}(t)\right) = + \infty$. Since the condition {\bf (a)} implies that $\lim\limits_{s \rightarrow + \infty} g_2(s) = g_1(T)$, {\bf (c)} holds.

Now, suppose there are $\alpha_1, \alpha_2$ satisfying {\bf (a)}, {\bf (b)}, and {\bf (c)}.

First, consider the case where $\alpha_1([0,+\infty[ ) = \alpha_2([0,+\infty[ ) =[0,+\infty[ $.
Then, there is a sequence  $\{T_j \}_{j \in \mathbb N}$ such that
\[
\lim T_j = +\infty, \quad \text{and}  \quad  \alpha_i(T_j) < \alpha_i(T_{j+1}) \ \ \forall j \in \mathbb N, \ i =1,2 .
\]
For any sequence $\{ \varepsilon_j \in]0,1[ \}_{j \in \mathbb N}$, the functions
\begin{align*}
\alpha_i^\varepsilon (t) = \sum_{j=1}^\infty \Bigg( &
\alpha_i(T_{j-1}) +
(1-\varepsilon_j) (\alpha_i(t) - \alpha_i(T_{j-1}) ) +
\\ & +
\varepsilon_j \frac{\alpha_i(T_j) - \alpha_i(T_{j-1}) }{T_j-T_{j-1}}(t-T_{j-1})
\Bigg)
\chi_{[T_{j-1},T_j[}(t)
\end{align*}
belong to $\mathcal T$ and therefore $\alpha_2^\varepsilon \circ \left( \alpha_1^\varepsilon \right)^{-1} \in \mathcal T$. Also,
\[
\left| \alpha_i^\varepsilon (t) - \alpha_i (t) \right| \leq
\sum_{j=1}^\infty
\varepsilon_j \left| \alpha_i(T_j) - \alpha_i(T_{j-1})  \right| \chi_{[T_{j-1},T_j[}(t) \qquad
\forall t \geq 0 .
\]
Since $g_1, g_2$ are uniformly continuous in compact intervals, for every $\delta >0$ there is some sequence $\{ \varepsilon_j \in]0,1[ \}_{j \in \mathbb N}$ such that
$
\left\| g_1 \circ \alpha_1^\varepsilon - g_2 \circ \alpha_2^\varepsilon \right\|_{L_\infty[0,+\infty[} < \delta $.
Since
$
\left\| g_1 \circ \alpha_1^\varepsilon - g_2 \circ \alpha_2^\varepsilon \right\|_{L_\infty[0,+\infty[} =
\left\| g_1 - g_2 \circ \alpha_2^\varepsilon \circ \left(\alpha_1^\varepsilon \right)^{-1}  \right\|_{L_\infty[0,+\infty[}
$, we see that \eqref{Eq orientation} holds.

In the case where $\alpha_1([0,+\infty[ ) =[0,+\infty[ $ and $ \alpha_2(\infty ) = T < +\infty$, there is a sequence  $\{T_j \}_{j \in \mathbb N}$ such that
\[
\lim T_j = +\infty, \quad \text{and}  \quad  \alpha_1(T_j) < \alpha_1(T_{j+1}) \ \ \forall j \in \mathbb N.
\]
Then we can apply a similar argument to the functions
\begin{align*}
\alpha_1^\varepsilon (t) = \sum_{j=1}^\infty \Bigg( &
\alpha_1(T_{j-1}) +
(1-\varepsilon_j) (\alpha_1(t) - \alpha_1(T_{j-1}) ) +
\\ & +
\varepsilon_j \frac{\alpha_1(T_j) - \alpha_1(T_{j-1}) }{T_j-T_{j-1}}(t-T_{j-1})
\Bigg)
\chi_{[T_{j-1},T_j[}(t), \\
\alpha_2^\varepsilon (t) = \alpha_2(t) &+ \varepsilon_1 t ,
\end{align*}
and this completes the proof.
\end{proof}

\subsection{Proof of Lemma~\ref{L AC reparameterization}}
\label{SP L AC reparameterization}

\begin{proof}
To prove {\bf (a)}:
Pick $t \in  \left] \alpha(0), \alpha(\infty) \right[$, and let $\hat \theta = \alpha^\#(t)$.
By continuity of $\alpha$, there is some $s \in ]0,+\infty[$ such that $t=\alpha(s)$, and $\hat \theta= \max \left\{ \theta: \alpha(\theta) = \alpha (s) \right\}$, that is, $\alpha (\hat \theta) = \alpha (s) =t$.
The equality $\dot \alpha \circ \alpha^\#(t) =0$ reduces to $\dot \alpha (\hat \theta) =0$.
Therefore, $\dot \alpha \circ \alpha^\#(t) =0$ implies $t \in \alpha \left(\left\{ \theta: \dot \alpha =0 \right\} \right)$.
Since this set has zero Lebesgue measure, we proved {\bf (a)}.

To prove {\bf (b)} and {\bf (c)}:
Let $A= \left\{ t: \dot \alpha (t) = 0, \ \dot g (t) \neq 0 \right\}$, and let $\mu$ denote the Lebesgue measure.

For each $\varepsilon >0$ there is a sequence of intervals $\left\{ ]a_i,b_i[ \right\}_{i \in \mathbb N}$ such that
\[
A \subset \bigcup_{i=1}^\infty]a_i,b_i[, \qquad
\sum_{i=1}^\infty (b_i-a_i) < \mu(A)+\varepsilon , \qquad
\sum_{i=1}^\infty \left(\alpha(b_i)-\alpha(a_i) \right) < \varepsilon .
\]
Fix $\varepsilon$ and a sequence as above and let
\[
t_i= \alpha (b_i), \quad
s_i=\alpha(a_i) - \frac{\varepsilon}{2^i}, \qquad
i \in \mathbb N .
\]
Notice that $\sum\limits_{i=1}^\infty (t_i-s_i) < 2 \varepsilon $ and $\alpha^\#(t_i) \geq b_i$, $\alpha^\#(s_i)< a_i$ for every $i \in \mathbb N$.
Therefore,
\begin{align*}
\sum_{i=1}^\infty \int_{\alpha^\#(s_i)}^{\alpha^\#(t_i)} | \dot g(s)| ds \geq
\sum_{i=1}^\infty \int_{a_i}^{b_i} | \dot g(s)| ds \geq
\int_A | \dot g(s)| ds .
\end{align*}
Thus, $g \circ \alpha^\#$ cannot be absolutely continuous when $\mu (A)>0$.

Now, suppose that $\mu (A) = 0$. In order to prove that $g \circ \alpha^\#$ is absolutely continuous and satisfies \eqref{Eq derivative of reparameterized curve}, we only need to consider the case where $g$ is scalar. Taking the decomposition $g=g^+-g^-$, where $g^+(t) = g(0) + \int_0^t \max \left(0, \dot g(s)\right) ds$ and $g^-(t) = \int_0^t \max \left(0, - \dot g(s)\right) ds$, we only need to consider the case where $g:[0,+\infty[ \mapsto \mathbb R $ and $\dot g (s) \geq 0 $ for a.e. $s \geq 0$.

Fix $t_1< t_2$ with $t_1 \geq 0$, $t_2 < \alpha(\infty) $, and fix $T \in ] \alpha^{\#}(t_2), + \infty[$.
For each $i \in \mathbb N$, let
\[
\alpha_i(s) = \sup\left\{ \alpha (\tilde s) + \frac{s-\tilde s}{i} : \tilde s \in [0,s] \right\} \qquad \forall s \in [0,T] ,
\]
and let $B_i= \left\{ s \in [0,S], \ \alpha(s) < \alpha_i(s) \right\} $.

Since $\dot g (s) = 0$ for a.e. $s \in \left[  \alpha^{\#}(t^-) ,  \alpha^{\#}(t) \right]$, and $\lim\limits_{i \rightarrow \infty} \alpha_i^{-1}(t) = \alpha^{\#}(t^-)$, we have
\begin{align*}
 & g \circ \alpha^{\#}(t_2) - g \circ \alpha^{\#}(t_1)= \int_{ \alpha^{\#}(t_1^-)}^{ \alpha^{\#}(t_2^-)} \dot g ds =
\lim_{i \rightarrow \infty} \int_{ \alpha_i^{-1}(t_1)}^{ \alpha_i^{-1}(t_2)} \dot g ds = \\
= &
\lim_{i \rightarrow \infty} \left(
 \int_{ [\alpha_i^{-1}(t_1), \alpha_i^{-1}(t_2)] \setminus B_i} \dot g ds +
  \int_{ [\alpha_i^{-1}(t_1), \alpha_i^{-1}(t_2)] \cap B_i} \dot g ds
\right) .
\end{align*}
Since $ \bigcap\limits _{i \in \mathbb N} B_i \setminus \{ \dot \alpha =0 \} $ is a set of zero Lebesgue measure, it follows that
\begin{align*}
 & g \circ \alpha^{\#}(t_2) - g \circ \alpha^{\#}(t_1)=
\lim_{i \rightarrow \infty}
\int_{ [\alpha_i^{-1}(t_1), \alpha_i^{-1}(t_2)] \setminus B_i} \dot g ds  .
\end{align*}
Notice that $\alpha_i$ is absolutely continuous and
$
\dot \alpha_i = \dot \alpha \chi_{B_i^c} + \frac 1 i \chi_{B_i} \geq \frac 1 i
$.
Therefore,
\begin{align*}
g \circ \alpha^{\#}(t_2) - g \circ \alpha^{\#}(t_1)= &
\lim_{i \rightarrow \infty}
\int_{ [t_1, t_2] \setminus \alpha_i(B_i)} \frac{\dot g}{\dot \alpha_i} \circ \alpha_i^{-1} ds  = \\
=&
\lim_{i \rightarrow \infty}
\int_{ [t_1, t_2] \setminus \alpha_i(B_i)} \frac{\dot g}{\dot \alpha} \circ \alpha^{\#} ds .
\end{align*}
Since $B_{i+1} \subset B_i$ and the set $\alpha_i(B_i)$ has Lebesgue measure no greater that $\frac T i$, the Lebesgue monotone convergence theorem guarantees that
\[
g \circ \alpha^{\#}(t_2) - g \circ \alpha^{\#}(t_1)=
\int_{t_1}^{ t_2} \frac{\dot g}{\dot \alpha} \circ \alpha^{\#} ds .
\]
Thus, $g \circ \alpha^\#$ is absolutely continuous and satisfies \eqref{Eq derivative of reparameterized curve}.

To prove {\bf (d)}:
Notice that $\alpha (s) = \alpha(t) $ for every $s \in \left[ t, \alpha^\# \circ \alpha(t) \right]$. Since the set $\{ t: \dot \alpha (t) =0, \ \dot g (t) \neq 0 \}$ has zero Lebesgue measure, we see that  $g (s) = g(t) $ for every $s \in \left[ t, \alpha^\# \circ \alpha(t) \right]$. In particular, $g\circ \alpha^\# \circ \alpha(t) = g(t)$. Thus, the result follows from Lemma \ref{L orientation}.
\end{proof}

\subsection{Proof of Proposition \ref{P lift of functions of bounded variation}}
\label{SP P lift of functions of bounded variation}

\begin{proof}
Without loss of generality, we can assume that $x$ has finite variation in the interval $[0,T]$ and it is constant in $[T, + \infty[$.

Consider a sequence of partitions of the interval $[0,T]$
\[
P_k = \left\{ 0 =t_{k,0} < t_{k,1} < \cdots < t_{k,k} =T \right\} \qquad k \in \mathbb N,
\]
such that $P_k \subset P_{k+1}$ for every $k \in \mathbb N$, and $\bigcup\limits_{k \in \mathbb N} P_k$ is dense in $[0,T]$.
Let $x_k:[0,+\infty[ \mapsto \mathbb R^n$ be the piecewise linear function interpolating the points $x(t_{k.i})$, $i=0,1, \ldots k$ and $x_k(t) =x(T) $ for every $t>T$. Then, $\left\{ (\theta_k,y_k)=\left( \ell_{(t,x_k)}^{-1}, x_k \circ \ell_{(t,x_k)}^{-1} \right) \right\}_{k \in \mathbb N} $ is a sequence in $\mathcal Y_n$.

The length of the graph of $x_k$ on the interval $[0,T]$ is
\begin{align*}
\ell_{(t,x_k)}(T)= &
\sum_{i=1}^k\sqrt{(t_i-t_{i-1})^2 + \left|x(t_i)-x(t_{i-1})\right|^2} \leq
T+{\rm V}_{[0,T]}(x) .
\end{align*}
Thus, the sequence $\left\{ \left( \theta_k,y_k \right) \right\} $ is uniformly bounded and equicontinuous on the interval $\left[0,T+{\rm V}_{[0,T]}(x) \right]$, and the Ascoli-Arzel\`a theorem guarantees that it has a subsequence converging uniformly uniformly towards some $(\theta,y) \in \mathcal Y_n$.
Without loss of generality, we assume that this subsequence is $\{ (\theta_k,y_k)\}$.

Notice that $\left\{ \theta_k^{-1}(t) \right\}$ may fail to converge towards $\theta^\#(t)$ if $t$ is a discontinuity point of $\theta^\#$.
Instead, we take the sequence
$\tilde \theta_k =
\left(
\theta_k-
\left\| \theta_k - \theta \right\|_{L_\infty[0,T+{\rm V}_{[0,T]}(x) ]}
\right)^+ $.
Notice that $\left\{ (\tilde \theta_k,y_k) \right\}$ converges uniformly towards $(\theta,y)$ and $\lim \limits_{k \rightarrow \infty} \tilde \theta _k^\#(t) = \theta^\#(t)$ for every $t \in [0,T ]$. Therefore,
\[
\lim_{k \rightarrow \infty} y_k \circ \tilde{\theta}_k^\#(t) = y \circ \theta^\# (t) \qquad \forall t \in [0,T] .
\]

Now, suppose that $x$ is continuous at the point $t \in [0,T]$.
By continuity, for every $\varepsilon>0$ there is some $\delta>0$ such that $|x(\tau )- x(t)|< \varepsilon $ for every $\tau \in \left] t - \delta, t + \delta \right[ $.
This implies $|x_k(\tau) -x(t) | < \varepsilon $ for every sufficiently large $k$ and every $\tau \in \left] t- \frac \delta 2, t+ \frac \delta 2 \right[$, because then $x_k(\tau)$ is a convex combination of points in $B_\varepsilon(x(t))$.
Thus,
\[
y \circ \theta^\# (t) =
\lim_{k \rightarrow \infty} y_k \circ \tilde \theta_k^\# (t) =
\lim_{k \rightarrow \infty} x_k\left(t+ \| \theta_k-\theta \|_{L_\infty[0,T+{\rm V}_{[0,T]}(x)]} \right) =
x(t) .
\]
\end{proof}

\subsection{Proof of Lemma~\ref{L convexity of reduced lagrangean}}
\label{SP L convexity of reduced lagrangean}

\begin{proof}
Notice that
\begin{align*}
&
L \left( y, \frac{\lambda w +(1-\lambda ) \hat w}{\lambda v +(1-\lambda ) \hat v} \right) \left( \lambda v +(1-\lambda ) \hat v \right) =
\\ = &
L \left( y,
\frac{\lambda v}{\lambda v +(1-\lambda ) \hat v} \frac w v +
\frac{(1-\lambda) \hat v}{\lambda v +(1-\lambda ) \hat v} \frac{\hat w}{\hat v}
\right) \left( \lambda v +(1-\lambda ) \hat v \right) \leq
\\ \leq &
\left(
\frac{\lambda v}{\lambda v +(1-\lambda ) \hat v} L \left( y, \frac w v \right) +
\frac{(1-\lambda) \hat v}{\lambda v +(1-\lambda ) \hat v} L \left( y, \frac{\hat w}{\hat v} \right)
\right)
\left( \lambda v +(1-\lambda ) \hat v \right) =
\\ = &
\lambda L \left( y, \frac w v \right) v +
(1-\lambda)  L \left( y, \frac{\hat w}{\hat v} \right) \hat v .
\end{align*}
Therefore, $(v,w) \mapsto L \left( y, \frac w v \right) v$ is convex.

The inequality
\begin{align*}
& \liminf_{\scriptsize \begin{array}{c}
(y,v,w) \rightarrow (\hat y, \hat v, \hat w) \\
v>0
\end{array}} L\left( y, \frac w v \right) v \leq
\\ \leq &
\liminf_{\scriptsize \begin{array}{c}
(v,w) \rightarrow (\hat v, \hat w) \\
v>0 
\end{array}} L\left( \hat y, \frac w v \right) v \leq
\liminf_{v \rightarrow \hat v, \ v>0} L\left( \hat y, \frac {\hat w} v  \right) v
\end{align*}
holds trivially. Therefore, we only need to prove that
\begin{align}
&
\limsup_{v \rightarrow \hat v, \ v>0} L\left( \hat y, \frac {\hat w} v  \right) v
\leq
\liminf_{\scriptsize \begin{array}{c}
(y,v,w) \rightarrow (\hat y, \hat v, \hat w) \\
v>0
\end{array}} L\left( y, \frac w v \right) v .
\label{Z001}
\end{align}
Due to continuity of $L$, this inequality holds for every $\hat v >0$. Suppose $\hat v =0$ and fix $b< \limsup\limits_{v \rightarrow 0^+} L \left( \hat y, \frac{\hat w}{v} \right) v $, and $\varepsilon >0$.
Then, we can pick $a \in ]0, \varepsilon]$ such that
$L \left( \hat y , \frac{ \hat w}{ a} \right) > b \frac{1}{a}$.
By continuity of $L$, there is some $\delta >0$ such that
\begin{equation}
\label{Z025}
L\left( y, \frac{w}{a} \right) > \frac{b}{a},
\quad \text{and} \quad
\left| L(y,w)- L(\hat y, \hat w) \right| < \varepsilon
\end{equation}
for every $(y,w)$ such that $|y- \hat y| < \delta$ and  $|w-\hat w| < \delta$.

Due to convexity of $w \mapsto L(y,w)$, we have
\begin{align*}
L\left( y, \frac w v \right) \geq &
L(y,w) + \frac{L\left( y, \frac{w}{a}\right) - L(y,w)}{\frac{1}{a}-1 } \left( \frac 1 v -1 \right) =
\\ = &
L(y,w) + \frac{a}{1-a} \left( L\left( y, \frac{w}{a}\right) - L(y,w)\right) \frac{1-v}{v} \qquad
\forall v \in ]0, a] .
\end{align*}
Using the estimates \eqref{Z025}, this yields
\begin{align*}
L \left( y, \frac w v \right) v \geq &
\left( L \left( \hat y, \hat w \right) - \varepsilon \right) v + \frac{a}{1-a} \left( \frac b a -L\left(\hat y, \hat w \right) - \varepsilon \right) (1-v) =
\\ = &
\frac{1-v}{1-a}b + L\left( \hat y , \hat w \right) \left( v-a \frac{1-v}{1-a} \right) - \varepsilon \left( v+a \frac{1-v}{1-a} \right) ,
\end{align*}
that is,
\[
\liminf_{\scriptsize \begin{array}{c}
(y,v,w) \rightarrow (\hat y, \hat v, \hat w) \\
v>0
\end{array}} L\left( y, \frac w v \right) v  \geq
\frac{1}{1-a}b - \left( L\left( \hat y , \hat w \right) +\varepsilon \right) \frac{a}{1-a}  .
\]
Making $\varepsilon$ tend to zero and $b$ tend to  $\limsup\limits_{v \rightarrow 0^+} L \left( \hat y, \frac{\hat w}{v} \right) v $, this implies \eqref{Z001}.
\end{proof}

\subsection{Proof of Proposition \ref{P Filippov}}
\label{SP P Filippov}

\begin{proof}
Fix  $(C,\theta,y)$, a trajectory of the differential inclusion \eqref{Eq nonparametric dynamics}, and let $V_t= \left( \dot \theta(t), \dot y(t) \right)$ for almost every $t \in [0,T]$.

For each compact set $K \subset \mathbb R^{1+k}$, consider the function $F_K:[0,T] \mapsto \overline{\mathbb R}$, defined defined almost everywhere by
\[
F_K(t) = \inf \left\{ \lambda (y(t),v,w) : (v,w) \in B^+ \cap K , \left( v,f(y(t))v+G(y(t))w \right) =V_t \right\},
\]
being understood that $\inf \emptyset = + \infty$.

First, we show that the functions $F_K$ are measurable.

For any set $A \subset \mathbb R^k$ and any $\varepsilon >0$, let $B_\varepsilon(A) = \bigcup\limits_{x \in A} B_\varepsilon(x)$.
Then, lower semicontinuity of $\lambda$ implies that for any $\alpha \in \mathbb R$,
\begin{align*}
&
F_K^{-1}\left( ]-\infty, \alpha [ \right) =
\\ = &
\left\{ t:
\exists (v,w) \in B^+ \cap K ,
\left( v,f(y(t))v+G(y(t))w \right) =V_t,
\lambda(y(t),v,w) < \alpha
\right\} =
\\ = &
\bigcap_{i \in \mathbb N}
\begin{array}[t]{l}
\Big\{ t:
\exists (v,w) \in B_{\frac 1 i} \left( B^+ \cap K \right),
\left| \left( v,f(y(t))v+G(y(t))w \right) -V_t \right|< \frac 1 i, \\
\hspace{2cm}
\lambda(y(t),v,w) < \alpha
\Big\} .
\end{array}
\end{align*}
Due to Lemma \ref{L convexity of reduced lagrangean}, this is
\begin{align*}
&
F_K^{-1}\left( ]-\infty, \alpha [ \right) =
\\ = &
\bigcap_{i \in \mathbb N}
\bigcup_{v \in \mathbb Q \cap ]0,1]}
\begin{array}[t]{l}
\Big\{ t:
\exists w \in \mathbb R^k,
(v,w) \in B_{\frac 1 i} \left( B^+ \cap K \right) , \\
\hspace{0.3cm}
\left| \left( v,f(y(t))v+G(y(t))w \right) -V_t \right|< \frac 1 i,
L\left( y(t),\frac w v \right)v < \alpha
\Big\} .
\end{array}
\end{align*}
Due to continuity of $L$, this further reduces to
\begin{align*}
&
F_K^{-1}\left( ]-\infty, \alpha [ \right) =
\\ = &
\bigcap_{i \in \mathbb N}
\bigcup_{v \in \mathbb Q \cap ]0,1]}
\bigcup_{{\scriptsize \begin{array}{c}w \in \mathbb Q^k: \\ (v,w) \in B_{\frac 1 i} \left( B^+ \cap K \right)\end{array}}}
\begin{array}[t]{l}
\Big\{ t:
\left| \left( v,f(y(t))v+G(y(t))w \right) -V_t \right|< \frac 1 i, \\
\hspace{1cm}
L\left( y(t),\frac w v \right)v < \alpha
\Big\} .
\end{array}
\end{align*}
Since $y$, $V$ are measurable and $f,G,L$ are continuous, it follows that $F_K$ is measurable.

Now, we construct a sequence $\{ \mathcal A_i \}_{ i \in \mathbb N }$ with the following properties:
\begin{itemize}
\item[{\rm (a)}]
Each $\mathcal A_i = \{ A_{i,1}, A_{i,2}, \ldots , A_{i,h_i} \}$ is a finite ordered collection of measurable subsets of $B^+ $;
\item[{\rm (b)}]
All the members of each collection $\mathcal A_i$ are pairwise disjoint, $B^+ = \bigcup \limits _{A \in \mathcal A_i} A$, and each element of $\mathcal A_i$ is contained in a ball of radius $\frac 1 i$;
\item[{\rm (c)}]
For any $i<j$, every element of $\mathcal A_j$ is a subset of some element of $\mathcal A_i$. All elements of $\mathcal A_j$ that are contained in $A_{i,h}$ precede (in the order of $\mathcal A_j$) any element of $\mathcal A_j$ contained in $A_{i,h+1}$.
\end{itemize}
To see that such sequences exist, let $\mathcal A_0=\left\{ B^+ \right\}$. For each $i \in \mathbb N$, let $\mathcal B_i = \{ B_{i,1},B_{i,2}, \ldots , B_{i,j_i} \}$, a finite cover of $B^+ $ by balls of radius $\frac 1 i$, and let
\[
C_{i,h} = B_{i,h} \setminus \bigcup_{l<h}B_{i,l}, \qquad h=1,2, \ldots , j_i .
\]
For each $i \in \mathbb N$, let $\mathcal A_i$ be the collection of intersections
\[
A \cap C_{i,h}, \qquad A \in \mathcal A_{i-1}, \quad h = 1,2, \ldots , j_i ,
\]
ordered in any way such that any $C_{i,h}\cap A_{i-1,l}$ precedes every $C_{i,s}\cap A_{i-1,l+1}$ (discard empty intersections).
So, $\{ \mathcal A_i \}_{i \in \mathbb N }$ satisfies (a)--(c).

Fix a sequence $\{ \mathcal A_i \}_{i \in \mathbb N }$ as above and for each $i \in \mathbb N$, $j \in \{ 1, 2, \ldots , j_i\}$, fix $(v,w)_{i,j}=(v_{i,j}, w_{i,j}) \in A_{i,j}$.
For each $i \in \mathbb N$, consider a function $j(i, \cdot ):[0,T] \mapsto \mathbb N$ defined almost everywhere  by
\begin{equation}
\label{Z002}
j(i,t) = \min \left\{ h\in \{ 1,2, \ldots , j_i\}: F_{\overline{A_{i,h}}}(t)=F_{B^+ }(t) \right\} ,
\end{equation}
and consider the sequence $\{(v_i,w_i): [0,T] \mapsto B^+ \}_{i \in \mathbb N}$ defined as
\[
(v_i,w_i)(t)=(v,w)_{i,j(i,t)}, \qquad i \in \mathbb N , \ t \in [0,T].
\]
Notice that $(v_i,w_i)([0,T])\subset \left\{ (v,w)_{i,j}, j \in \{ 1, 2, \ldots , j_i\} \right\}$ is a finite set and
\begin{align*}
&
\{ t: (v_i,w_i)(t)=(v,w)_{i,j} \} =
\left\{ t: F_{\overline{A_{i,j}}}(t)=F_{B^+ }(t) \right\} \setminus \bigcup_{h<j}
\left\{ t: F_{\overline{A_{i,h}}}(t)=F_{B^+ }(t) \right\} .
\end{align*}
Therefore, measurability of $F_K$ guarantees measurability of $(v_i,w_i)$.

For almost every $t \in [0,T]$, we have:
\begin{itemize}
\item[]
$\left(v_i(t),f(y(t))v_i(t) + G(y(t))w_i(t) \right)= V_t \qquad \forall i \in \mathbb N$, and
\item[]
$ \{ (v_i,w_i)(t) \}_{i \in \mathbb N} $ is a Cauchy sequence.
\end{itemize}
Thus,  $(v,w)(t)= \lim\limits_{i \rightarrow \infty} (v_i,w_i)(t)$ is a measurable function satisfying
\[
\left(v(t),f(y(t))v(t) + G(y(t))w(t) \right)= V_t \qquad \text{a.e. }t \in [0,T].
\]
Lower semicontinuity of $\lambda$ and \eqref{Z002} imply that
\begin{align*}
&
\lambda (y(t),v(t),w(t)) =
\\ = &
\inf\left\{
\lambda (y(t),\tilde v, \tilde w): (\tilde v, \tilde w) \in B^+ , \left(\tilde v, f(y(t))\tilde v +G(y(t)) \tilde w \right)=V_t
\right\} \leq
\\ \leq &
\dot C(t)
\end{align*}
for almost every $t \in [0,T]$.
\end{proof}

\section*{ACKNOWLEDGMENTS}

The research of the first coauthor has been supported by  FCT--Funda\c c\~ao para a Ci\^encia e Tecnologia (Portugal) via  strategic project PEst-OE/EGE/UI0491/2013, he is grateful to INDAM (Italy) for supporting his visit to the University of Florence in January 2014. The research of the second coauthor has been supported by MIUR (Italy) via national project (PRIN) 200894484E of MIUR (Italy); he is also grateful to CEMAPRE (Portugal) for supporting his research stay at  ISEG,   University of Lisbon in May 2013.

\end{document}